\documentclass[11pt, twoside]{amsart}

\title[NIM-reps of TY generalizations]{NIM-representations of Tambara-Yamagami generalizations}

\author[A. Czenky]{Agustina Czenky}
\address{Czenky: Department of Mathematics, University of Southern California,
3620 S. Vermont Ave., KAP 104. Los Angeles, CA 90089, USA}
\email{czenky@usc.edu}

\author[E. McGovern]{Emily McGovern}
\address{McGovern: Department of Mathematics, University of Oregon, 1021 E. 13th Ave, Eugene, OR 97403, USA}
\email{emilycmcgovern@gmail.com}

\author[M. Molander]{Melody Molander}
\address{Molander: Department of Mathematics, The Ohio State University, 100 Math Tower, 231 W 18th Ave, Columbus, OH 43210, USA}
\email{molander.3@osu.edu}

\author[M. Müller]{Monique Müller}
\address{Müller: Departamento de Matem\'atica e Estat\'istica, Universidade Federal de S\~ao Jo\~ao del-Rei, Pra\c ca Frei Orlando, 170, S\~ao Jo\~ao del-Rei, MG 36307-352, Brazil}
\email{monique@ufsj.edu.br}

\author[A. Ros Camacho]{Ana Ros Camacho}
\address{Ros Camacho: Facultat de Ciències Matemàtiques, Universitat de València, Avinguda Vicent Andrés Estellés 19, 46100 Burjassot (València), Spain}
\email{ana.ros.camacho@uv.es}

\usepackage{amsmath}
\usepackage{amsfonts}
\usepackage{amsthm}
\usepackage{amssymb,bbm}
\usepackage{dsfont}
\usepackage{stmaryrd}
\usepackage{mathabx}
\usepackage{appendix}
\usepackage[english]{babel}
\usepackage{url}
\usepackage{fancyhdr}
\usepackage{graphicx}
\usepackage{verbatim}
\usepackage[normalem]{ulem}
\usepackage[shortlabels]{enumitem}
\newcommand{\stkout}[1]{\ifmmode\text{\sout{\ensuremath{#1}}}\else\sout{#1}\fi}
\usepackage[
colorlinks=true,
linkcolor=black, 
anchorcolor=black,
citecolor=black,
urlcolor=black, 
]{hyperref}
\usepackage{cleveref}
\usepackage[all,cmtip]{xy}
\usepackage{geometry}
\usepackage[dvipsnames]{xcolor}
\usepackage{quiver}
\usepackage{lscape} 
\usepackage{multicol}
\usepackage{float}

\usepackage[utf8]{inputenc}
\usepackage{tikz}
\usepackage{tikz-cd}

\allowdisplaybreaks

\usepackage{adjustbox}

\usepackage{enumitem}

\usepackage[notcite,notref,final]{showkeys}

\usepackage{calrsfs}
\DeclareMathAlphabet{\cal}{OMS}{zplm}{m}{n}

\DeclareMathAlphabet{\mathsf}{OT1}{cmss}{m}{n} 

\newcommand{\id}{\textnormal{id}}
\newcommand{\one}{\mathbbm{1}}

\newcommand{\Hom}{\textnormal{Hom}}

\usepackage{relsize}

\newcommand{\act}{\triangleright}

\newcommand{\Vect}{\mathsf{Vec}}

\newcommand{\GLM}{\textnormal{GLM}(\Gamma, \delta)}

\newcommand{\mZ}{\mathbb{Z}}

\newcommand{\cC}{\cal{C}}

\numberwithin{equation}{section}

\newtheorem{theorem}{Theorem}[section]

\newtheorem{proposition}[theorem]{Proposition}

\newtheorem{corollary}[theorem]{Corollary}
\newtheorem{lemma}[theorem]{Lemma}

\newtheorem{theorem*}{Theorem}

\theoremstyle{definition}
\newtheorem{definition}[theorem]{Definition}

\newtheorem{thma}{Theorem}

\newtheorem{example}[theorem]{Example}
\newtheorem{remark}[theorem]{Remark}

\makeatletter
\@namedef{subjclassname@2020}{%
  \textup{2020} Mathematics Subject Classification}
\makeatother



\tikzcdset{scale cd/.style={every label/.append style={scale=#1},
    cells={nodes={scale=#1}}}}

\begin{document}

\begin{abstract}
 We compute and classify the irreducible non-negative integer matrix (NIM-)represen\-tations of two proposed generalizations of the Tambara-Yamagami fusion ring, as studied by Jordan-Larson and Galindo-Lentner-Möller, respectively. We also detect the algebra objects associated to these NIM-representations.
\end{abstract}
\maketitle

\section{Introduction}
Tambara-Yamagami (TY) categories form one of the simplest families of non-pointed fusion categories: each is determined by a finite abelian group $A$, a non-degenerate symmetric bicharacter $\chi: A \times A \to \mathbb{C}^\times$, and a sign $\tau = \pm |A|^{-1/2}$ \cite{TY}. Invertible objects are labelled by elements of $A$, and there is a single non-invertible simple object $X$ satisfying $X \otimes X = \bigoplus_{a \in A} a.$

Despite their elementary description, TY categories play a central role in mathematical physics. 
Via the Turaev-Viro \cite{TV} or the Reshetikhin-Turaev \cite{RT} constructions, they produce $\left(2+1 \right)$-dimensional TQFTs that have been studied intensively (see e.g. \cite{DS,TV} regarding their associated invariants).

TY categories are also realized in conformal field theory: even-rank TY categories arise via $\mathbb{Z}_2$-orbifolds of conformal nets \cite{B}. Their simple fusion rules and braiding data allow explicit computation of topological invariants and modular data such as their $S$ and $T$ matrices, making TY categories ideal toy models for phenomena such as anyon condensation, symmetry fractionalization, and defect excitations \cite{DS,TW1, TW2, CZ}. In short, they are well-studied and provide a bridge between algebraic category theory and concrete models of topological phases and defects.

Motivated by this, it is natural to ask how Tambara-Yamagami categories fit into broader families of fusion categories. Studying extensions of TY fusion rings therefore can provide a controlled framework to explore the new phenomena that arise when one moves just beyond the TY setting. In particular, generalized TY fusion rules can interpolate between pointed, Ising-like, and more complex near-group behaviors, while often retaining enough structure, such as almost-group-like tensor products, to allow for explicit computations.  In this paper, we focus on two extensions introduced in \cite{JL} and \cite{GLM}; each generalizes the TY fusion rules in a distinct way.

The extension $R_{p,G}$ introduced in \cite{JL} has as invertible basis elements a group, $G$, of square order and $p-1$ non-invertible elements for fixed positive integer $p$. Under suitable conditions, the resulting fusion category is a $\mathbb Z_p$-extension of $\Vect_{\Gamma}$, where each non-trivial component contains exactly one non-invertible simple object; setting $p=2$ recovers a Tambara-Yamagami category. In contrast, in \cite{GLM} the fusion ring $\GLM$ has as invertible basis elements an abelian group $\Gamma$ and as non-invertible basis elements, $\Gamma/2\Gamma$.The resulting category is a non-degenerate braided $\mathbb{Z}_2$-crossed extension of $\Vect_{\Gamma}$ of the form $\Vect_{\Gamma}\oplus \Vect_{\Gamma/2\Gamma}$, where the $\mathbb Z_2$-action on $\Gamma$ is given by $-\id$. When $\Gamma$ has odd order, this construction recovers a Tambara-Yamagami category; thus the focus of the present paper is the even-order case.

One useful tool in the study of fusion rings is the classification of their \emph{non-negative integer matrix representations}, or NIM-reps, for short. Their importance is partly motivated by their role in the boundary rational conformal field theory, where finding NIM-reps is equivalent to finding solutions to Cardy’s equation \cite{Behrend}. Mathematically, NIM-reps can be used as an effective method to detect algebra objects in the fusion category $\mathcal{C}$ underlying the fusion ring: algebra objects can be read off directly from the combinatorial data of a NIM-graph \cite{HR}. This provides insight into the module categories that can be constructed over $\mathcal{C}$, since algebra objects classify module categories \cite{Ost} and therefore play a central role in understanding extensions, orbifolds, and other constructions in tensor categories. Furthermore, they are key ingredients in classification problems in mathematical physics, such as full conformal field theories \cite{FRS} or extensions of vertex operator algebras \cite{CKM,HKL}, to mention a couple.

In this paper, we study the NIM-reps  of the generalizations of the Tambara-Yamagami fusion rings introduced in \cite{JL, GLM}, classify the irreducible ones, and obtain the shape of their algebra objects. This provides a first suggestion for the classification of the module categories over the corresponding fusion categories.  We note that, since our fusion rings are generalizations of TY, we recover the classification results for NIM-reps over TY \cite{HR}.
We include below a summary of our main results. 

\begin{thma}\label{thma}
For the Jordan-Larson fusion ring $R_{p,G}$, 
\begin{itemize}
  \item[(a)]  [Theorem \ref{prop:p divides n of orbits}]
 Let $M$ be an irreducible NIM-rep over $R_{p,G}$, and  let $m$ be the number of orbits of the associated $G$-action. Then $m$ divides $p$. 
    \item[(b)]  [Theorem \ref{prop: p-orbits RpG}] Let $p>2$ and let $m$ be a divisor of $p$. Irreducible NIM-reps over $R_{p,G}$ with $m$ orbits are parametrized by subgroups $H_1, \dots, H_m<G$ such that $|H_i||H_{\sigma_k(i)}|/|G|$ is a square integer for all $1\leq i,k\leq p-1$, where $\sigma_k:=(1 \dots m)^k$.
    \item[(c)] [Theorem \ref{thm:classification RpG}]    Two irreducible NIM-reps $M(H_1, \dots, H_m)$ and $M(H'_1, \dots, H'_m)$ are isomorphic if and only if there is a permutation $\tau\in S_m$  such that 
         $H_i$ is conjugate to $H'_{\tau(i)}$.
\end{itemize}
\end{thma}

\begin{thma}\label{thmb}
For the Galindo-Lentner-Möller fusion ring GLM$(\Gamma,\delta)$, 
\begin{itemize}
    \item[(a)] [Theorem \ref{prop:atmost2}] 
  An irreducible NIM-rep over $\GLM$ has at most two $\Gamma$-orbits. 
  \item[(b)] [Theorem \ref{thm: 1 Gamma orbit}]
   Irreducible NIM-reps over GLM$(\Gamma,\delta)$ with a unique $\Gamma$-orbit are parametrized by the choice of a subgroup $H<\Gamma$ such that $\sqrt{|2\Gamma|}$ divides $|H\cap 2\Gamma|$, and the choice of $\tau_0$ satisfying Proposition \ref{prop:tau0}.
    \item[(c)] [Theorem \ref{thm:irrepNIMs-2Gammaorb}] Irreducible NIM-reps over $\operatorname{GLM}(\Gamma,\delta)$ with two $\Gamma$-orbits are parametrized by
   \begin{enumerate}[leftmargin=*,label=\rm{(\roman*)}] 
       \item\label{item-1-thm:irrepNIMs-2Gammaorb} A choice of subgroups $H_1, H_2<\Gamma$ such that $\delta \in \overline{H}_1\cap \overline{H}_2$ or $\delta \not\in \overline{H}_1\cup \overline{H}_2$, $|(\Gamma/H_1)[2]|=|(\Gamma/H_2)[2]|$, and  $\sqrt{|H_1\cap 2\Gamma| |H_2\cap 2\Gamma|/|2\Gamma|}$ is an integer. 
       \item\label{item-2-thm:irrepNIMs-2Gammaorb} A choice of $\tau_0$ satisfying Propositions \ref{prop:tau0} and Lemma \ref{rk: switch of gamma orbits}.
   \end{enumerate}
  \item[(d)] [Theorems \ref{thm: class unique Gamma orbit} and \ref{thm: class two gamma orbits}] Two irreducible NIM-reps $M(\{H_i\}_{i=1}^k,\tau_0)$ and $M(\{H_i'\}_{i=1}^k,\tau_0)$ are isomorphic if and only if (up to reordering) $H_k=H_k'$ and $\tau_0=\tau_0'$, where $k$ is the number of $\Gamma$-orbits.
\end{itemize}
\end{thma}

Proofs of the results of Theorem A stated above can be found in Section \ref{sec:JordanLarson}. The proofs of the results of Theorem B can be found in Section 4. 

We note that Theorem \ref{thma}[(a)] and Theorem \ref{thmb}[(a)] point to a strong connection between the grading of a fusion category and the number of orbits permitted in an irreducible NIM-rep. Indeed, the Jordan–Larson and Galindo–Lentner–Möller fusion rings categorify to $\mathbb Z_p$- and $\mathbb Z_2$-graded fusion categories, respectively \cite{GLM,JL}. We believe these results highlight a deeper structural connection that merits further investigation in future work.

This paper is organized as follows. In Section \ref{sec:basicnotions}, we recall some basic terminology for this project. In Sections \ref{sec:JordanLarson} and \ref{sec:GLM}, we compute and classify the irreducible NIM-reps of each fusion ring, and detect the corresponding algebra objects for each case.

\subsection*{Acknowledgements} 
The authors thank the organizers of the Women in Mathematical Physics 3 research and mentoring program from which this collaboration was formed. A. Czenky\ was partially supported by the Simons Collaboration Grant No. 999367. E. McGovern was partially supported by NSF Grant  DMS-2039316.
 M. M\"uller thanks the excellent working conditions at the Department of Mathematics at Indiana University Bloomington, where part of this research was done. M. M\"uller was partially supported by Simons Foundation Award 889000 as part of the Simons Collaboration on Global Categorical Symmetries. 
 A. Ros Camacho thanks Cardiff University and their University Research Leave Scheme 2025/26 for support.

\section{Basic notions}\label{sec:basicnotions}

We start by setting some notation conventions:  $S_m$ denotes the permutation group of order $m!$,  and $\mZ_{+}$ the semi-ring of positive integers with zero. We assume familiarity of the reader with the notions introduced in \cite[Chapter 3 and 7]{EGNO}. Some references for NIM-reps are \cite{Behrend,GannonNIM,Gannon,Gannonbook}.

Let $(R,B)$ be a fusion ring with basis $B=\lbrace b_i \rbrace_{i \in I}$. A \textit{non-negative integer matrix representation}  (\text{NIM-rep}),  of $(R,B)$ is a $\mZ_+$-module $(T,M)$ (where M = $\{m_\ell\}_{\ell \in L}$ is a $\mathbb{Z}$-basis for $T$) that satisfies the \emph{rigidity condition}:
 let T have a symmetric bilinear form $(-,-):M \times M \to \mZ$ defined by $(m_\ell,m_k)=\delta_{\ell,k}$,
for any $\ell,k \in L$.
Then we must have, for any $i\in I, \ell,k\in L$, 
\[
(b_i\act m_\ell,m_k) = (m_\ell,b_{i^*}\act m_k),
\]
where $\act$ denotes the action of the basis $B$ on the $\mathbb{Z}_+$-module, and $^*$ denotes the involution of the fusion ring as a based ring. We may refer to a NIM-rep simply as $M$ instead of $\left( T,M \right)$.

Note that, for $t \in T, m_\ell \in M$, the symmetric bilinear form $(t,m_\ell)$ counts the multiplicity of $m_\ell$ in the basis decomposition of $t$. Then $(t,t)=1$ if and only if $t \in M$.

A \textit{NIM-rep morphism} between two NIM-reps $\left( T,M \right)$ and $\left( T',M' \right)$ over a given fusion ring $\left( R,B \right)$ is a function $\psi \colon M \to M'$ inducing a $\mZ$-linear map between the modules. If $\psi$ is a bijection, and the induced map is an isomorphism of $R$-modules, then we say that the NIM-reps are \textit{equivalent}, see \cite{DB}. Also, the \textit{direct sum} of two such NIM-reps is the $R$-module $T \oplus T'$ with  basis $M \oplus M'$. Similarly, we can define \textit{sub-NIM-rep}. A NIM-rep is called \textit{irreducible} if it has no proper sub-NIM-reps. See \cite[ \S 3.4]{EGNO}, \cite[Lemma 2.1]{Ost2}.

Given a NIM-rep $(T,M)$ over a fusion ring $(R,B)$, a \textit{NIM-graph} (or  \textit{fusion graph}) is constructed with a node for each element of the basis $M$, and a directed arrow with source $m_\ell$ and target $m_k$, labeled by an element $b_i \in B$, for every copy of $m_k$ in $b_i \vartriangleright m_\ell$.  Every node in a NIM-graph will have a self-loop labeled by the ring identity, which we omit sometimes for simplicity. The NIM-graph allows us to visualize the irreducibility of the corresponding NIM-rep, as a NIM-rep is irreducible if and only if the NIM-graph is connected.

In this paper, we introduce a new combinatorial tool for the study of NIM-reps, defined below.

\begin{definition} Let $R$ be a fusion ring with basis $B$ given by the union of a set of invertible elements forming a group $G$ and some non-invertible elements. The \textit{NIM-orbit graph} of a NIM-rep $M$ of $R$ shows the action of the non-invertible elements, while obscuring the action of the elements of $G$. We get the NIM-orbit graph from the NIM-graph by first contracting all edges labeled by elements of the group (this process results in a single vertex of the graph for each orbit under the $G$-action). If after this step there are multiple edges between two $G$-orbits labeled by a non-invertible element, we delete all but one of these edges.  
\end{definition}

By collapsing the action of the invertible elements, the NIM-orbit graph isolates the essential combinatorial structure coming from the action of the non-invertible objects. In the present work, this plays a central role in organizing and simplifying the classification arguments. We expect that the same approach can be adapted to other fusion rings with a non-trivial invertible component, particularly in classification problems where the fusion ring has a graded structure.

\begin{example}\label{ex: Jordan-Larson NIM graph}
    Suppose we have a Jordan-Larson fusion ring $R_{p,G}$ as defined in Section \ref{sec:JordanLarson}, with $G=\mathbb{Z}_2\times \mathbb{Z}_2$ and $p=4$. The basis of this ring are the elements of $G$ and three non-invertibles: $X_1, X_2, X_3$. Suppose the action of $G$ is determined by the stabilizer subgroup $H=\langle (1,0)\rangle$. Example \ref{ex:Z2xZ2p4NIM-orbitgraph} shows that 
     this induces a 2-dimensional representation where $X_k \rhd m_j= m_1 + m_2$ for $1\leq k \leq 3$ and $j=1,2$. The NIM-graph associated to this case would then be:
          $$
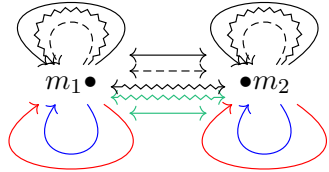
\begin{tikzcd}[column sep=4em]
m_1 \bullet
  \arrow[loop,dashed, distance=2.4em, in=125, out=55]
  \arrow[loop, rightsquigarrow, distance=3.6em, in=135, out=45]
  \arrow[loop, distance=5.0em, in=145, out=35]
  \arrow[color=red, loop, distance=5.2em, in=-145, out=-35]
  \arrow[color=blue, loop, distance=2.8em, in=-125, out=-55]
  \arrow[r, leftrightarrow, shift left=10pt, shorten >=1.5ex, shorten <=1.5ex]
  \arrow[r, leftrightarrow, shift left=4pt, dashed, shorten >=1.5ex, shorten <=1.5ex]
  \arrow[r, leftrightsquigarrow, shift right=6pt,
        color={rgb,255:red,39;green,190;blue,125}]
  \arrow[r, leftrightarrow, shift right=12pt,
        color={rgb,255:red,39;green,190;blue,125}, shorten >=1.5ex, shorten <=1.5ex]
& \bullet m_2
  \arrow[loop,dashed, distance=2.4em, in=125, out=55]
  \arrow[loop, rightsquigarrow, distance=3.6em, in=135, out=45]
  \arrow[loop, distance=5.0em, in=145, out=35]
  \arrow[color=red, loop, distance=5.2em, in=-145, out=-35]
  \arrow[color=blue, loop, distance=2.8em, in=-125, out=-55]
   \arrow[l, leftrightsquigarrow, shift left=2pt]
\end{tikzcd}$$
where blue and red arrows indicate the elements in $H$, the green arrows indicate the elements in $G - H$, and the black arrows indicate the actions of the noninvertible elements $X_1, X_2, X_3$. If we collapse the NIM-graph onto its one orbit, the corresponding NIM-orbit graph would be:
$$
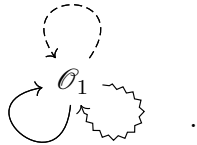
\begin{tikzcd}
\mathcal{O}_1 \arrow[loop, dashed, distance=3em, in=125, out=55] \arrow[loop, dashed, distance=3em, in=125, out=55] \arrow[loop, rightsquigarrow, distance=3em, in=-70, out=-5] \arrow[loop, distance=3em, in=-170, out=-95] \arrow[loop, distance=3em, in=-170, out=-95]
\end{tikzcd}\hbox{\raisebox{-0.5cm}{.}}$$
\end{example}

Finally, given a categorifiable fusion ring, we can construct a \textit{fusion category} - that means, a semi-simple finite tensor category. Inside these categories, we recall the following objects: an \textit{algebra} in a fusion category $\cC=\left( \cC,\otimes, \one,\alpha,\lambda,\rho \right)$ is a triple $\left( A,m,u \right)$, where $A \in   \cC $, and  $m \colon A \otimes A \to A$ (the multiplication), $u \colon \mathbbm{1} \to A$ (the unit) are two morphisms in $\cC$,
    satisfying unitality and associativity constraints, see \cite[Definition 7.8.1]{EGNO}.

There has been some efforts on performing exhaustive classifications of algebra objects in diverse families of fusion categories with different methods, \cite{HLRC,Nat,Ost,WINART} to mention a few. As mentioned in the introduction, it is a well known result \cite{Ost2} that classifying algebra objects in fusion categories is the same as classifying some module categories over these fusion categories. Explicitly, if $\mathcal M$ is an indecomposable semisimple $\mathcal C$-module category, and $N\in \mathcal M$ is such that it generates the Grothendieck group of $\mathcal M$, $\text{Gr}(\mathcal M)$, as a based $\mathbb Z_+$-module over the Grothendieck ring of $\mathcal C$, $\text{Gr}(\mathcal{C})$, then $\mathcal M$ is equivalent to the category $\operatorname{Mod}_{\mathcal C}-A$ of right $A$-modules in $\mathcal C$, where $A:=\underline{\Hom}(M,M)$ denotes the internal Hom space of $M$ \cite{Ost2}. 
 
 This result can be translated to NIM-reps language in the following way. Suppose  $\text{Gr}(\mathcal M)$, generated as a NIM-rep of $\text{Gr}(\mathcal{C})$ by a certain element $m$. Given another basis element of this Grothendieck group, say $m_i$, then there is a basis element of the fusion ring $b_j$ satisfying that $b_j \act m=m_i$. This inspires the following definition:

\begin{definition}\label{NIMadmis}\cite[Definition 2.36]{HR}
    A NIM-rep $\left(A,M \right)$ over the fusion ring $\left( R,B \right)$ of a fusion category is called \textit{admissible} if there exists some $m_0 \in M$ such that for any $m_i \in M$, there exists some $b_j \in B$ such that $b_j \rhd m_0=m_i$. 
\end{definition}

Given such an admissible NIM-rep, then it is possible \cite[Proposition 2.36]{HR} to construct an algebra object by taking $\textstyle{\bigoplus_{i \in I} a_i b_i}$ where $a_i$ is the number of self-loops labeled by $b_i$ in the NIM-graph of the NIM-rep.
Moreover, if it is the underlying NIM-rep of $\operatorname{Mod}_{\mathcal C}-A$ as in the setup for \cite{Ost2}, then $\textstyle A=\bigoplus_{i \in I} a_i b_i$. Hence, while it does not provide a complete description of the structure maps, this procedure  
 provides an intuitive way to detect these algebras thoroughly in any given fusion category, and a first suggestion towards the classification of the module categories coming from these fusion rings.

\section{Jordan-Larson generalization of Tambara-Yamagami}\label{sec:JordanLarson}

In this section, we study the fusion ring $R_{p,G}$ as defined in \cite{JL} by Jordan and Larson, classifying all NIM-reps over these rings  and detecting the algebra objects associated to the NIM-reps. 

\begin{definition} \cite{JL}
    Let $G$ be a finite group of square order and let $p$ be a positive integer. The fusion ring $R_{p,G}$ has as basis the elements of $G$ and $X_1, \dots, X_{p-1}$, with the relations
  \begin{align*}
      &g\, X_i = X_i = X_i \, g, &X_i^*=X_{p-i},\\
      &X_i\, X_j= \begin{cases}
      \sqrt{|G|} X_{i+j}   &  \text{if } i+j \ne p\\
       \textstyle  \oplus_{g\in G} g  & \text{if } i+j = p.
      \end{cases}
  \end{align*}
\end{definition}

Here we note that the index $i$ in $X_i$ is always taken modulo $p$. The assumption that $|G|$ is a square can be dropped when $p=2$. The ring $R_{2,G}$ is the Tambara-Yamagami fusion ring. In this sense, for $p>2$ these rings are generalizations of the Tambara-Yamagami fusion rings. 

In \cite{JL} the authors also show that when $(|G|,p)=1$, categorifications of $R_{p,G}$ can be parame\-trized by group data \cite[Proposition 3.1]{JL}. They also note that any categorification of $R_{p,G}$ will be $\mathbb Z_p$-graded. For the particular case $p=3$, they are able to fully classify the categorifications of $R_{3,G}$ for any $G$ such that $(|G|,3)=1$ \cite[Theorem 1.4]{JL}, thus obtaining a family of $\mathbb Z_3$-graded fusion categories generalizing the $\mathbb Z_2$-graded Tambara-Yamagami categories.

\subsection{Preliminary results for irreducible NIM-reps over \texorpdfstring{$R_{p,G}$}{RpG}}\label{sec: preliminary RpG}
In this section, we set up some notation for working with a NIM-rep over $R_{p,G}$, and analyze the number of orbits the associated $G$-action can have.

The action of $R_{p,G}$ on a NIM-rep $M$ consists of the action of both the group elements and the non-invertible elements $X_1, \dots, X_p.$ If we restrict ourselves to the group component, we obtain a $G$-action on the NIM-rep basis \cite{HR}. We can then partition $M$ into $G$-orbits as follows:
\begin{align}\label{eq:number of orbits}
  \textstyle   M= \bigsqcup_{i=1}^m \{m_\ell^i\}_{1\leq \ell \leq |G:H_i|}, 
\end{align}
where the $i$ index is moving around the $m$ distinct orbits of the $G$-action, each of which is defined by a subgroup $H_i<G$. We will denote the corresponding orbit by $\mathcal O_i:=\{m_\ell^i\}_{1\leq \ell\leq |G:H_i|}$.

The following result mimics \cite[Proposition 3.9]{HR} and its proof is analogous.
\begin{lemma} \label{lemma: action on orbits-JL}
 Let $M$ be an irreducible NIM-rep over $R_{p,G}$, and fix an orbit label $i$. Then for any $1\leq k \leq p-1$, we have $X_k \rhd m_{\ell_1}^i = X_k \rhd m_{\ell_2}^i$, for all $1\leq \ell_1, \ell_2\leq |G:H_i|$.
\end{lemma}

We will thus denote $c_{i,j}^k:=(X_k \rhd m_\ell^i, m_t^j),$ dropping the $\ell$ and $t$ labels. Note also that by the rigidity condition, $c_{i,j}^{k}=c_{j,i}^{p-k}.$ With this notation, we have that
\begin{align*}
   \textstyle X_k \rhd m_\ell^i =\sum_{j=1}^n c_{i,j}^k \sum_{\ell=1}^{|G:H_j|} m_\ell^j,
\end{align*}
and so acting on both sides by $X_{p-k}$ we obtain
\begin{align*}
\textstyle |H_i| \sum_{t=1}^{|G:H_i|} m_t^i=   \sum_{g\in G} g \rhd m_\ell^i = \sum_{j=1}^n c_{i,j}^k |G:H_j| X_{p-k} \rhd m_\ell^j.
\end{align*}
Counting the multiplicities of the $i$-labeled and $q$-labeled orbits on either side, for $q\ne i$, we get
\begin{align}\label{eq: coeff counter}
    &\textstyle |H_i|=\sum_{j=1}^n c_{i,j}^k c_{j,i}^{p-k} |G:H_j|, &\text{ and }
  &&\textstyle 0= \sum_{j=1}^n c_{i,j}^k c_{j,q}^{p-k} |G:H_j|.
\end{align}

Since the summands are all non-negative integers, this implies that for all $j$ and $i\ne q$, 
\begin{align}\label{eq: coeff condition}
    c_{i,j}^kc_{j,q}^{p-k}=0.
\end{align}

\begin{lemma}\label{lemma:non-zero c}
For a fixed $X_k$ and a fixed orbit  $i$, there exists an orbit $j$ such that $c_{i,j}^k\ne 0.$
\end{lemma}
\begin{proof}
    This follows directly from the definition of $c_{i,j}^k,$ as $X_k\rhd m_\ell^i$ must be a non-zero sum of basis elements in $M$.
\end{proof}

\begin{lemma}\label{lemma:unique orbit}
 For a fixed orbit $1\leq i \leq p$ and fixed $1\leq k\leq p-1,$ the summands in $X_k \rhd m_{\ell}^i$  are contained in a unique orbit. 
\end{lemma}
\begin{proof}
By Lemma \ref{lemma:non-zero c}, there exists an index $j$ such that $0\ne c_{i,j}^k=c_{j,i}^{p-k}.$ Equation \eqref{eq: coeff condition} then implies that $c_{j,q}^{p-k}=0$ for all $q\ne i,$  and so the summands in $X_{p-k}\rhd m_s^j$ are all contained in $\mathcal O_i$. It follows from acting with $X_k$ and Lemma \ref{lemma: action on orbits-JL} that the summands in $X_k \rhd m_\ell^i$ are contained in $\mathcal O_j$.
\end{proof}

The following result is immediate from the previous lemma. 

\begin{corollary}\label{corollary:connected to p-1 orbits}
In the NIM-graph, any orbit $\mathcal O_i$ is connected to at most $p-1$ orbits.
\end{corollary}

\begin{remark}\label{rk: X1 does not fix orbits} Suppose the action of $R_{p,G}$ on an irreducible NIM-rep $M$ has more than one orbit. Then
   the object $X_1$ cannot fix any orbit, i.e. the summands in $X_1\rhd m^i_\ell$  must be contained in $\mathcal O_j$ with $j\ne i$. This follows since if $X_1$ fixes the orbit $\mathcal O_i,$ then by the fusion rules $X_t$ can be generated by $X_1^{t}$ and thus also fixes $\mathcal O_i$ for all $t,$ which contradicts irreducibility.
\end{remark}

\begin{remark}\label{rk: Xk does not fix orbits}
    By the same argument as the previous remark, if $(k,p)=1$ then $X_k$ cannot fix any orbit in an irreducible NIM-rep. 
\end{remark}
The next result allows us to link the number of an irreducible NIM-rep over $R_{p,G}$ to  $p$.

\begin{theorem}\label{prop:p divides n of orbits}
 Let $M$ be an irreducible NIM-rep over $R_{p,G}$, and  let $m$ be the number of orbits of the $G$-action. Then $m$ divides $p$. 
\end{theorem}
\begin{proof}
    Suppose $m>1$. By Lemma \ref{lemma:unique orbit}, we have an action of $X_1$ on the set of orbits $\{\mathcal O_1, \dots, \mathcal O_m\},$ defined as follows: $X_1 \cdot \mathcal O_i:=\mathcal O_j$, where $\mathcal O_j$ is the unique orbit containing the summands of $X_1\rhd m_\ell^i$ for any $\ell$. We can then identify this action of $X_1$ with a permutation $\sigma \in S_m$. 

    Consider the factorization $\sigma=\sigma_1\dots\sigma_\ell$ of $\sigma$ into disjoint cycles. The irreducibility condition on the NIM-rep then implies that $\ell=1$, since the orbits in the NIM-rep must be connected to each other. Hence $\sigma$ is a cycle and, by Remark \ref{rk: X1 does not fix orbits}, its length is exactly $m$. 

    Lastly, since by the fusion rules $X_1^{p}$ is a sum of invertible elements, we must have that $|\sigma|=m$ divides $p$, as desired. 
\end{proof}

\begin{remark}\label{remark:cycle} 
    If $m>1$, following the  proof of Theorem \ref{prop:p divides n of orbits} we can visualize the action of $X_1$ on the set of orbits $\{\mathcal O_1, \dots, \mathcal O_m\}$ with the following diagram
    \begin{equation}
        \begin{tikzcd}
	{\mathcal O_1} & {\mathcal O_2} & {\mathcal O_3} & \cdots & {\mathcal O_m}
	\arrow["{X_1}", curve={height=-12pt}, from=1-1, to=1-2]
	\arrow["{X_1}", curve={height=-12pt}, from=1-2, to=1-3]
	\arrow["{X_1}",  curve={height=-12pt}, from=1-3, to=1-4]
	\arrow["{X_1}", curve={height=-12pt}, from=1-4, to=1-5]
	\arrow["{X_1}", curve={height=-12pt}, from=1-5, to=1-1].
\end{tikzcd}
    \end{equation}
\end{remark}

\subsection{Classification of NIM-reps over \texorpdfstring{$R_{p,G}$}{RpG}} \label{sec:classification RpG}
In this section, we give a classification of irreducible NIM-reps over $R_{p,G}$ with $m$ orbits for any $m$ that divides $p$, as per Theorem \ref{prop:p divides n of orbits}. 

\begin{theorem}\label{prop: p-orbits RpG}
     Let $p>2$ and let $m$ be a divisor of $p$. Irreducible NIM-reps over $R_{p,G}$ with $m$ orbits are parametrized by subgroups $H_1, \dots, H_m<G$ such that $|H_i||H_{\sigma_k(i)}|/|G|$ is a square integer for all $1\leq i,k\leq p-1$, where $\sigma_k:=(1 \dots m)^k$.
\end{theorem}
\begin{proof}
 Let $M$ be an irreducible NIM-rep over $R_{p,G}$ with $m$ orbits coming from the $G$-action, and let $H_1, \dots, H_m<G$ be the corresponding subgroups parametrizing the orbits. 
 Recall from Remark \ref{remark:cycle} that $X_1$ acts as a cycle $\sigma_1$ of size $m$ on the set of orbits, and so we may relabel the orbits to $\{\mathcal O_1, \dots, \mathcal O_m\}$ so that the $X_1$ action is given by $\sigma_1=(1 \dots m)$.
Then the action of $X_k$ on the set of orbits is given by  $\sigma_k=\sigma_1^k=(1 \ \dots \ m)^k.$
 Following Equation \eqref{eq: coeff counter} we obtain 
    \begin{align*}
    &|H_i|= c_{i,\sigma_k(i)}^kc_{\sigma_k(i),i}^{p-k} |G:H_{\sigma_k(i)}|=(c_{i,\sigma_k(i)}^k)^2 |G:H_{\sigma_k(i)}| ,
    \end{align*}
 and so 
    \begin{align*}
        (c_{i,\sigma_k(i)}^k)^2= |H_i||H_{\sigma_k(i)}|/|G|,
    \end{align*}
    for all $i=1, \dots, m$ and $k=1, \dots, p-1.$ It follows then that 
 $ c_{i}^k=|H_i||H_{\sigma_k(i)}|/|G|$ is a square integer for all $i$ and $k$. 

 Reciprocally, suppose we have a choice of subgroups $H_1, \dots, H_m<G$ such that $|H_i||H_{\sigma_k(i)}|/|G|$ is a square integer for all $i$ and $k$. This choice of subgroups defines a $G$-action over $M$ with $m$ orbits. To see that this lifts to a NIM-rep action over $R_{p,G}$, we need to define an action of $X_1, \dots, X_{p-1}$ on $M$ which respects the fusion rules. From Lemma \ref{lemma: action on orbits-JL}, it is enough to determine the numbers $c_{i,j}^k=(X_k \rhd m_\ell^i, m_t^j)$ for all $1\leq i,j,k\leq p-1$. For a fixed $1\leq k \leq p-1$, we set $c_{i,\sigma_k(i)}^k=\sqrt{|H_i||H_{\sigma_k(i)}|/|G|},$ and zero everywhere else. It is a quick check that this definition satisfies the fusion rules and thus gives a well-defined NIM-rep. 
\end{proof}

\begin{remark}\label{rmk:formula-c-JL} As shown in the proof above, the corresponding action is defined by setting    $$\textstyle X_k \rhd m_{\ell}^{i}= \sqrt{|H_i||H_{\sigma_k(i)}|/|G|}\sum_{j=1}^{|G:H_{\sigma_k(i)}|}  m_j^{\sigma_k(i)},$$ for all $k=1, \dots, p-1,$ $1\leq i \leq m$ and $1\leq \ell \leq |G:H_i|$.
\end{remark}

\begin{example}
    Irreducible NIM-reps over $R_{p,G}$ with a unique orbit are parametrized by subgroups $H<G$ such that $\sqrt{|G|}$ divides $|H|$. Moreover,  $$\textstyle X_k \rhd m_l^{1}= (|H|/\sqrt{|G|})\sum_{j=1}^{|G:H|}  m_j^{1},$$ for all $k=1, \dots, p-1.$ That is, $c^k_{1,1}=c^r_{1,1}$ for all $1\leq k,r \leq p-1.$
\end{example}

We denote by $M(H_1, \dots, H_m)$ the NIM-rep obtained from each possible choice of subgroups $H_1, \dots, H_m<G$, as described in Theorem \ref{prop: p-orbits RpG}.

\begin{theorem}\label{thm:classification RpG}
        Two irreducible NIM-reps $M(H_1, \dots, H_m)$ and $M(H'_1, \dots, H'_m)$ are isomorphic if and only if there is a permutation $\tau\in S_m$  such that 
         $H_i$ is conjugate to $H'_{\tau(i)}$.
\end{theorem}

\begin{proof}
 If $M(H_1, \dots, H_m)$ and $M(H'_1, \dots, H'_m)$ are isomorphic as NIM-reps, then they are also isomorphic as $G$-sets, which implies that  there is a permutation $\tau\in S_m$  such that 
         $H_i$ is conjugate to $H'_{\tau(i)}$. On the other hand, suppose we have two irreducible NIM-reps  $M(H_1, \dots, H_m)$ and $M(H'_1, \dots, H'_m)$ and a permutation $\tau\in S_m$ such that   $H_i$ is conjugate to $H'_{\tau(i)}$. To show that they are isomorphic as NIM-reps it is enough to show that the induced isomorphism $\psi: M(H_1, \dots, H_m)\xrightarrow{\sim} M(H'_1, \dots, H'_m)$ of $G-$sets respects the action of $X_k$ for all $1\leq k \leq p-1.$ But this follows directly from Remark \ref{rmk:formula-c-JL} since 
$$\sqrt{|H_i||H_{\sigma_k(i)}|/|G|}=\sqrt{|H'_{\tau(i)}||H'_{\tau(\sigma_k(i))}|/|G|}$$ for all $k$ . 
\end{proof}

\begin{example}
    Irreducible NIM-reps over $R_{p,G}$ with a unique orbit are classified by conjugacy classes of subgroups $H<G$ such that $\sqrt{|G|}$ divides $|H|$.
\end{example}

We illustrate our classification results with some concrete examples.

\begin{example}\label{ex:Z2xZ2p3NIM-orbitgraph} 
 Consider the case where $G = \mathbb Z_2 \times \mathbb{Z}_2 = \langle (0,1), (1,0)\rangle$ and $p = 3$. By Theorem \ref{prop:p divides n of orbits}, an irreducible NIM-rep can have only one or three orbits coming from the $G$-action. 
 We first consider NIM-reps with a unique orbit. Since we require $|H|^2/|G|$ to be a square, we have two choices: either $|H| = 2$, or $|H| = 4$.  Choosing $|H| = 2$ gives a 2-dimensional NIM-rep where $c_{1,1}^k = 1$ for $k \leq p-1$. There are three non-isomorphic 2-dimensional NIM-reps, corresponding to the choices of $\langle (0,1)\rangle$, $\langle (1,0)\rangle$  and $\langle (1,1)\rangle$ for $H$.  Choosing $|H| = 4$ gives a 1-dimensional NIM-rep where $c_{1,1}^k = 2$  for $k \leq p-1$.  Notice that this parameterization here is done independent of $p$, and so more generally, these are the only irreducible NIM-reps with one orbit for $R_{p, G}$. 
 
We will now look for NIM-reps with three orbits under the $G$-action. These are determined by the choice of three subgroups $H_1, H_2$ and $H_3$, where $\frac{|H_i||H_j|}{|G|}$ must be a square. We have three choices: either $|H_1| = 1 $ and $ |H_2| = |H_3| = 4$, or $|H_i| = 2$, for all $i$, or  $|H_i| = 4$ for all $i$. With the first and second choice, the respective NIM-reps are 6-dimensional, and with the third choice, the NIM-rep is 3-dimensional. There are 10 non-isomorphic NIM-reps, one in the first case, eight in the second case, and one in the third case. For the first case, by Remark \ref{rmk:formula-c-JL}, we have $c^1_{1,2}=1, c^1_{2,3}=2$, and $c^1_{3,1}=1$. Then, by rigidity and Equation \eqref{eq: coeff condition}, we have   \[c^1=\left(\begin{array}{ccc}
0 & 1 & 0 \\ 
0 & 0 & 2 \\ 
1 & 0 & 0
\end{array}\right)\quad \text{and}\quad c^2=\left(\begin{array}{ccc}
0 & 0 & 1 \\ 
1 & 0 & 0 \\ 
0 & 2 & 0
\end{array}\right).\] The coefficients $c_{i,j}^k$ for the other cases can be obtained in a similar way. In all cases,  the NIM-orbit graph is, up to re-ordering of the orbits, the following: 
   \begin{equation}
       \begin{tikzcd}
\mathcal O_1 \arrow[rr, "X_1", bend left] \arrow[rd, "X_2"] &                                                             & \mathcal O_2 \arrow[ld, "X_1", bend left] \arrow[ll, "X_2"'] \\
 & \mathcal O_3 \arrow[lu, "X_1", bend left] \arrow[ru, "X_2"] &                                                  \end{tikzcd}
   \end{equation}
These are all of the NIM-reps of $R_{3, \mathbb{Z}_2 \times \mathbb{Z}_2}$ with exactly 3 orbits. Since the number of orbits must be a divisor of $p = 3$, this example is a full classification of the irreducible NIM-reps of $R_{3, \mathbb{Z}_2 \times \mathbb{Z}_2}$.
\end{example}

\begin{example}\label{ex:Z2xZ2p4NIM-orbitgraph}
      Consider the case where $G = \mathbb Z_2 \times \mathbb Z_2$ and $p = 4$. Here, NIM-reps could have 1,2, or 4 orbits under the $G$-action.  The case where we have one $G$-orbit follows exactly like Example \ref{ex:Z2xZ2p3NIM-orbitgraph}. When $H=\langle(1,0)\rangle$, its NIM-graph is given in Example \ref{ex: Jordan-Larson NIM graph}. With two orbits we have the following three choices for orbit sizes: $|H_1| = 1$  and $|H_2| = 4$, or  $|H_1| = |H_2| = 2$  or $|H_1| = |H_2| = 4$. The first choice gives a 5-dimensional NIM-rep where $c_{ij}^k$ are 0 or 1. The second choice gives a 4-dimensional NIM-rep where $c_{ij}^k$ are 0 or 1. The third choice gives a 2-dimensional NIM-rep where $c_{ij}^{k}$ are 0 or 2. 
      Finally, when we have four orbits, we again have three cases: $|H_1| = 1, |H_2| = |H_3| = |H_4| = 4$, or all $|H_i| = 2$ or all $|H_i| = 4$.  We get the NIM-orbit graphs (up to reordering of orbits) found in  Table \ref{ex:Z2xZ2p4NIM-orbitgraphstable}.  

      \begin{table}[H]
          \centering
          \begin{tabular}{c|c|c}
            $\begin{tikzcd}
\mathcal{O}_1 \arrow[loop, dashed, distance=3em, in=125, out=55] \arrow[loop, dashed, distance=3em, in=125, out=55] \arrow[loop, rightsquigarrow, distance=3em, in=-70, out=-5] \arrow[loop, distance=3em, in=-170, out=-95] \arrow[loop, distance=3em, in=-170, out=-95]
\end{tikzcd}$ & $\begin{tikzcd}
\mathcal{O}_1 \arrow[ loop, rightsquigarrow, distance=2em, in=125, out=55] \arrow[rr, bend left] \arrow[rr, dashed, bend left=49] \arrow[rr, dashed, bend left=49] &  & \mathcal{O}_2 \arrow[ll, bend left=49] \arrow[loop, rightsquigarrow, distance=2em, in=55, out=125] \arrow[ll, dashed, bend left] \arrow[ll, dashed, bend left] 
\end{tikzcd}$  & 
$\begin{tikzcd}
        \mathcal{O}_1 \arrow[rr, dashed, bend left=9] \arrow[rr, dashed, bend left=9] \arrow[dd, bend left=9] \arrow[rrdd, rightsquigarrow] &&\mathcal{O}_2 \arrow[ll, bend left=9] \arrow[dd, dashed, bend left=9] \arrow[dd, dashed, bend left=9] \arrow[lldd,rightsquigarrow] \\&& \\
        \mathcal{O}_4 \arrow[uu, dashed, bend left=9] \arrow[uu, dashed, bend left=9] \arrow[rr, bend left=9] \arrow[uurr,rightsquigarrow] && \mathcal{O}_3 \arrow[uull,rightsquigarrow] \arrow[ll, dashed, bend left=9] \arrow[ll, dashed, bend left=9]\arrow[uu, bend left=9]
    \end{tikzcd}$
\\ \hline 
               1-orbit & 2-orbits & 4-orbits 
          \end{tabular}
          \caption{$G=\mathbb{Z}_2 \times \mathbb{Z}_2$, $p=4$, NIM-orbit graphs. Dashed arrows indicate action with $X_1$, resp. squiggle arrows with $X_2$ and plain arrows with $X_3$.}
          \label{ex:Z2xZ2p4NIM-orbitgraphstable}
      \end{table}
\end{example}

\subsection{Algebra objects from NIM-reps of \texorpdfstring{$R_{p,G}$}{RpG}}
In this section, we use our prior classification of irreducible NIM-reps over $R_{p,G}$ to prove all of them are admissible, and to construct the associated  algebra object in each case.

Let $M$ be an irreducible NIM-rep over $R_{p,G}$, with corresponding subgroups $H_1, \dots, H_m <G$ for $m|p$, detailed at Theorem \ref{prop: p-orbits RpG}. Remember that \cite[Proposition 2.36]{HR} if a NIM-rep is admissible, we will detect algebra objects at elements in the NIM-graph with non-trivial self-loops. When we rearrange a NIM-graph into a NIM-orbit graph, we are just reordering and collapsing the NIM-graphs and so self-loops will also appear if any orbit contains an element of the NIM-rep that has a self-loop in the NIM-graph. Again, if the NIM-rep is admissible, algebra objects hence will also be detected at these orbits with self-loops.

\begin{lemma}
Every irreducible NIM-rep over $R_{p,G}$  is admissible. 
\end{lemma}

\begin{proof}
    Take Remark \ref{remark:cycle}. We can use $X_1$ and powers thereof to connect the orbits with each other, and so we will always be able to connect two elements in different orbits.
\end{proof}

Take Example \ref{ex:Z2xZ2p3NIM-orbitgraph}. In this case, we cannot identify any self-loop, and so any algebra object will be group-like, given by $A_i=\oplus_{h \in H_i} h$ for any $1\leq i \leq 3$. However, observe Example \ref{ex:Z2xZ2p4NIM-orbitgraph} and remember Remark \ref{rk: Xk does not fix orbits}: there can be self-loops in certain orbits and these will be given by acting with some non-invertible whose label $k$ does divide $p$.

The self loops in the NIM-orbit graph can be predicted using the results developed earlier in this section. By Remark \ref{remark:action-X}, any set of orbits can be observed as a cycle by acting with $X_1$, and the action of $X_1$ can be identified with a permutation $\sigma \in S_m$. With this observation at hand, we have the following.

\begin{remark}
Recall Proposition \ref{prop:p divides n of orbits}, and take $p=m\ell$ for some $\ell$. If $X_1$ acts on the set of orbits as a permutation $\sigma=\left(1 2 3 \ldots m\right)$, then there will be self-loops via acting with $X_{jm}$ for $j=1, \dots, \ell-1$.
\end{remark}

From the previous remark and Theorem \ref{prop: p-orbits RpG}, analyzing the multiplicity of the self-loops we get the following.

\begin{proposition}
   Algebra objects associated to an irreducible NIM-rep over $R_{p,G}$ with $m$ orbits are of the form 
    $$A_i=\oplus_{h \in H_i} h \oplus_{j=1}^{\ell-1} c_{i,i}^{jm} X_{jm}=\oplus_{h \in H_i} h \oplus_{j=1}^{\ell-1} \sqrt{|H_i|^2/|G|} X_{jm},$$
    for all $1\leq i \leq m$. 
\end{proposition}

\begin{example}
    In the case where we have a unique orbit,  we obtain the following algebra object:
\[
    A= \oplus_{k=1}^{p-1} c_{1,1}^k X_k\oplus_{h \in H} h= \left((|H|/\sqrt{|G|}) \oplus_{k=1}^{p-1} X_k \right)\oplus_{h \in H} h.
\]
\end{example}

\section{Galindo-Lentner-Möller generalization of Tamabara-Yamagami fusion ring}\label{sec:GLM}

In this section, we study the fusion ring $\GLM$ as defined in \cite{GLM} by Galindo, Lentner and M\"oller, classify all NIM-reps over these rings, and detect the potential algebra objects associated to the NIM-reps. 

\begin{definition}
    \cite{GLM} Let $\Gamma$ be a finite abelian group of even order and $\delta\in\Gamma/2\Gamma$. The fusion ring $\operatorname{GLM}(\Gamma, \delta)$ has basis $\Gamma\cup \{X_{\overline{g}}\}_{\overline{g}\in \Gamma/2\Gamma}$, multiplication between group elements is the group operation,
    \begin{align*}
g\,X_{\overline{h}}=X_{\overline{g}+\overline{h}}=X_{\overline{h}}\,g, & & X_{\overline{g}}X_{\overline{h}}=\oplus_{t\in \Gamma, \overline{t}=\delta+\overline{g}+\overline{h}} t,
    \end{align*}
and the involution is given by $g^*=-g$ and $(X_{\overline{g}})^*=X_{-\overline{g}-\delta}$, for all $g,h\in \Gamma$. 
\end{definition}

A $\mathbb{Z}_2$-graded categorification for this fusion ring is constructed \cite[\S 5.2]{GLM}. In \cite{GLM0}, the authors show that if $\Gamma$ has odd order and $\mathbb{Z}_2$ acts by multiplication by -1 on $\Gamma$, then the resulting category is a Tambara-Yamagami fusion category.

\subsection{Preliminary results for irreducible NIM-reps over \texorpdfstring{$\GLM$}{GLM}}\label{sec: preliminary GLM}
In this section, we set up some notation for working with a NIM-rep over $\GLM$, and analyze the number of orbits the associated $\Gamma$- and $2\Gamma$-actions can have. 

Recall that, for a given NIM-rep $M$, if we restrict ourselves to the group component of $\operatorname{GLM}(\Gamma, \delta)$ we obtain a $\Gamma$-action on the NIM-rep basis \cite{HR}. So $M$ can be partitioned into $\Gamma$-orbits
\begin{align*}
  \textstyle  M=\bigsqcup_{j=1}^n \mathcal{M}^j,
\end{align*}
where the $j$ index is moving around the $n$ distinct $\Gamma$-orbits of the action, each of which is defined by a stabilizer subgroup $H_j<\Gamma$. 

It will be useful to us to further restrict this action to the subgroup $2\Gamma<\Gamma.$ Each $\Gamma$-orbit partitions into $2^m$ $2\Gamma$-orbits for some $m\geq 0.$ This follows since for a fixed orbit $\mathcal M^j$, the number of $2\Gamma$-orbits is $N_j=|\Gamma:2\Gamma|/|H_j: H_j\cap 2\Gamma|$. Since $\Gamma/2\Gamma$ is a group of torsion 2 and $H_j/(H_j\cap 2\Gamma)$ is a subgroup of $\Gamma/2\Gamma$, we find that $N_j$ must be a power of 2. Hence, we will partition each $\Gamma$-orbit into disjoint $2\Gamma$-orbits as follows. For each $j$, 
\begin{align*}
\textstyle \mathcal{M}^j=\bigsqcup_{i=1}^{N_j} \mathcal{O}_i^j,
\end{align*}
where $\mathcal{O}_i^j$ is a $2\Gamma$-orbit defined by $\mathcal{O}^j_i\coloneq \{m^{i,j}_\ell\}_{1\leq \ell \leq |2\Gamma:H_i^j|}$, and $H_i^j=H_j\cap 2\Gamma$. We will stick to this notation when referring to $\Gamma$ and $2\Gamma$-orbits from now on. To refer to the $i$th $2\Gamma$-orbit of the $j$th $\Gamma$-orbit, we write $(i,j)$ and call this an \emph{orbit pair}.

\begin{remark}
    An equivalent way to count the number of $2\Gamma$ orbits in each $\Gamma$-orbit $\mathcal M^i$ is $N_i=|(\Gamma/H_i)[2]|$, where for a group $G$ we define $G[2]:=\{g\in G : 2g=0\}.$
\end{remark}

\begin{lemma}\label{lemma: action on orbits-GLM} Let $M$ be an irreducible NIM-rep over $\operatorname{GLM}(\Gamma,\delta)$, and fix an orbit pair $(i, j)$. Then $X_{\overline{g}}\rhd m_{\ell_1}^{i,j}=X_{\overline{g}}\rhd m_{\ell_2}^{i,j}$ for all $1\leq \ell_1, \ell_2 \leq |2\Gamma: H_i^j|$, for all $\overline{g}\in \Gamma/2\Gamma$.
\end{lemma}

\begin{proof}
    As $m^{i,j}_{l_1}$ and $m^{i,j}_{l_2}$ are in the same $2\Gamma$-orbit, there is a group element $h\in 2\Gamma$ such that $h\rhd m^{i,j}_{l_1}=m^{i,j}_{l_2}$. Then using the module action and the fusion rules, we have that 
 $$
        X_{\overline{g}}\rhd m^{i,j}_{l_2}=X_{\overline{g}}\rhd (h\rhd m^{i,j}_{l_1})= (X_{\overline{g}}\,h)\rhd m^{i,j}_{l_1}= X_{\overline{g}+\overline{h}}\rhd m^{i,j}_{l_1}=X_{\overline{g}}\rhd m^{i,j}_{l_1}.$$
\end{proof}

\begin{remark} \label{remark:action-X}
    We will thus denote $c_{(i,j),(k, \ell)}^{\overline{g}}\coloneq (X_{\overline{g}} \rhd m_r^{i,j}, m_s^{k,\ell})$. By the rigidity condition, $c_{(i,j),(k,\ell)}^{\overline{g}}=c_{(k,\ell), (i,j)}^{-\overline{g}-\delta}$. With this notation, we have that
    \begin{align*}
          \textstyle      X_{\overline{g}}\rhd m_{r}^{i,j}=\sum_{\ell=1}^n\sum_{k=1}^{p_\ell}c_{(i,j),(k,\ell)}^{\overline{g}}\sum_{s=1}^{
        |2\Gamma:H_\ell\cap 2\Gamma|} m_s^{k,\ell}.
    \end{align*} So, acting on both sides by $X_{-\overline{g}-\delta}$, we obtain   
    \begin{align*}
  \textstyle      |H_j\cap 2\Gamma| \sum_{v=1}^{|2\Gamma: H_j\cap 2\Gamma|} m_v^{i,j}=\sum_{t\in 2\Gamma} t\rhd m_{r}^{i,j}=\sum_{\ell=1}^n\sum_{k=1}^{p_\ell}c^{\overline{g}}_{(i,j),(k,\ell)} \sum_{s=1}^{
        |2\Gamma:H_\ell\cap 2\Gamma|}   X_{-\overline{g}-\delta}\rhd m_{s}^{k,\ell}.
    \end{align*} Using Lemma \ref{lemma: action on orbits-GLM} and counting the multiplicities of the $(i,j)$-labeled and $(q,u)$-labeled orbits on either side, for $(q,u)\neq (i,j)$, we get the equations
\begin{align}
  |H_j\cap 2\Gamma|& \textstyle =\sum_{\ell=1}^n\sum_{k=1}^{p_\ell}(c^{\overline{g}}_{(i,j),(k,\ell)})^2|2\Gamma:H_\ell\cap 2\Gamma|, \text{ and }\label{eq:(1)} \\
    0&= \textstyle \sum_{\ell=1}^n\sum_{k=1}^{p_\ell} c_{(i,j),(k,\ell)}^{\overline{g}}c_{(k,\ell),(q,u)}^{-\overline{g}-\delta}|2\Gamma: H_\ell \cap2\Gamma|\label{eq:(2)}.
\end{align}
Since the summands on Equation \eqref{eq:(2)} are non-negative integers, this implies, for all orbit pair $(k,\ell)$, 
\begin{align}\label{eq:coeff-condition-GLM}
    c_{(i,j),(k,\ell)}^{\overline{g}}c_{(k,\ell),(q,u)}^{-\overline{g}-\delta}=0,
\end{align}
for all $(q,u)\neq (i,j)$. 
\end{remark}

The following lemma has proofs that mimic Lemma \ref{lemma:non-zero c}  and Lemma \ref{lemma:unique orbit}, so will be omitted.

\begin{lemma}\label{lemma:unique-orbit-non-invertible-GLM}
\begin{enumerate}[\upshape (a)]
    \item For $\overline{g}\in \Gamma/2\Gamma$ and an orbit pair $(i,j)$, there exists an orbit pair $(k,\ell)$ such that $c_{(i,j),(k,\ell)}^{\overline{g}}\ne 0$.
    \item   For an orbit pair $(i,j)$ and $\overline{g}\in\Gamma/2\Gamma$, the summands in $X_{\overline{g}} \rhd m_{\ell}^{i,j}$  are contained in a unique orbit pair. 
\end{enumerate}

\end{lemma}

\begin{corollary}\label{coro:unique-orbit-GLM}
    The summands in $X_{\overline{g}} \rhd \mathcal{M}^j$  are contained in a unique $\Gamma$-orbit.
\end{corollary}
\begin{proof}
     We need to see that the summands in $X_{\overline{g}} \rhd m_{\ell}^{i,j}$ and the summands in $X_{\overline{g}} \rhd m_{\ell'}^{i',j}$ are contained in the same $\Gamma$-orbit. Let $h\in \Gamma$ such that $h\rhd m_{\ell}^{i,j}=m_{\ell'}^{i',j}$. From Lemma \ref{lemma:unique-orbit-non-invertible-GLM}, the summands in $X_{\overline{g}} \rhd m_{\ell}^{i,j}$ are contained in a unique orbit pair $(s,k)$. Since the action of $h$ fixes any $\Gamma$-orbit, then the summands in $h\rhd (X_{\overline{g}} \rhd m_{\ell}^{i,j})$ are contained in a unique orbit pair $(s',k)$. By the fusion rules, $h\rhd (X_{\overline{g}} \rhd m_{\ell}^{i,j})=X_{\overline{g}}h \rhd m_{\ell}^{i,j}=X_{\overline{g}} \rhd (h\rhd  m_{\ell}^{i,j})=X_{\overline{g}} \rhd m_{\ell'}^{i',j}$. Therefore, the summands in $X_{\overline{g}} \rhd m_{\ell}^{i,j}$ and the summands in $X_{\overline{g}} \rhd m_{\ell'}^{i',j}$ are contained in the same $\Gamma$-orbit $k$.
\end{proof}

\begin{theorem}\label{prop:atmost2}
  An irreducible NIM-rep over $\GLM$ has at most two $\Gamma$-orbits. 
\end{theorem}

\begin{proof}
    Note that any two $\Gamma$-orbits  $\mathcal M^j$ and $\mathcal M^i$ must be connected by the action of $X_{\overline{0}}$. In fact, since the action is irreducible, there exists some $X_{\overline g}$ that sends $\mathcal M^j$ to $\mathcal M^i$. But, by the fusion rules, $X_{\overline g}=gX_{\overline 0}.$
    Since the action of $g$ fixes any $\Gamma$-orbit, it must be the case that $X_{\overline{0}}$ sends $\mathcal M^j$ to $\mathcal M^i$. Moreover, by Corollary \ref{coro:unique-orbit-GLM}, $X_{\overline{0}}$ acting on $\mathcal M^j$ is contained in a unique $\Gamma$-orbit, so there can only be at most two $\Gamma$-orbits. 
\end{proof}

\subsubsection{Action on the set of $2\Gamma$-orbits}\label{sec:2Gamma action}

The next series of lemmas show how the elements of the fusion ring $\operatorname{GLM}(\Gamma, \delta)$ act on the $2\Gamma$-orbits.

\begin{remark}
For any $g\in \Gamma$, we have that $2g\in 2\Gamma$ and so $g$ acts on the $2\Gamma$-orbits as a permutation $\sigma_g$ of order 1 or 2. Moreover, for any $g,h$ we have that $\sigma_g\sigma_h=\sigma_{g+h}$, and so if $\overline h=\overline g$, then $\sigma_g=\sigma_h.$
We set the notation $\sigma_{\overline{g}}:=\sigma_g$ from now on.
\end{remark}

We denote by $\tau_0$ the action on the 2$\Gamma$-orbits induced by $X_{\overline 0}$. 

\begin{lemma}\label{lemma:X-g in terms of sigmas}Let $M$ be a NIM-rep over $\operatorname{GLM}(\Gamma,\delta)$, and $\overline g\in \Gamma/2\Gamma$.  Then $X_{\overline{g}}$ acts on the set of $2\Gamma$-orbits by $\tau_0\sigma_{\overline g}$. 
\end{lemma}

\begin{proof}
   This follows by the fusion rules, since $X_{\overline g}=X_{\overline 0}g$, for any $g\in \Gamma$.
\end{proof}

\begin{proposition} \label{prop:tau0}
Let $M$ be a NIM-rep over $\operatorname{GLM}(\Gamma,\delta)$. Following the notation in this section, $\tau_0$ satisfies
  \begin{enumerate}[leftmargin=*,label=\rm{(\roman*)}] 
   \item\label{item:prop-i} $\tau_0^2=\sigma_{\delta}$, and
      \item\label{item:prop-ii} $\tau_0\sigma_{g}=\sigma_g\tau_0$, for all $g\in \Gamma$. 
  \end{enumerate} Moreover, given a $\Gamma$-action on $M$ then any choice of $\tau_0$ satisfying \ref{item:prop-i} and \ref{item:prop-ii} induces a well-defined  action of $\GLM$ on the (corresponding) set of $2\Gamma$-orbits. 
\end{proposition}
\begin{proof}
Items \ref{item:prop-i} and \ref{item:prop-ii} follow from the fusion rules in $\operatorname{GLM}(\Gamma, \delta).$ Conversely, we start with a $\Gamma$-action on $\mathcal M$ and 
    look at the induced action on the set of $2\Gamma$-orbits, i.e., by Lemma \ref{lemma:X-g in terms of sigmas}, we define the action of $X_{\overline g}$ on this set by $\tau_0\sigma_{\overline g}$, for all $g\in \Gamma$. We need to check that the action defined in this way satisfies all the fusion rules. To check that $gX_{\overline h}=X_{\overline g+\overline h}$ and $X_{\overline g+\overline h} = X_{\overline h}g$ are satisfied, we compute
    \begin{align*}
     &   \sigma_g \tau_0\sigma_{\overline h}=\tau_0\sigma_g\sigma_h=\tau_0\sigma_{\overline{g+h}}&\text{and} &&\tau_0\sigma_{\overline{g+h}}=\tau_0\sigma_{h+g}=\tau_0\sigma_{h}\sigma_g,
    \end{align*}
    respectively. For $X_{\overline{g}}X_{\overline{h}}=\sum_{t\in \Gamma, \overline{t}=\delta+\overline{g}+\overline{h}} t$, we need to check that $\tau_0\sigma_{\overline g}\tau_0\sigma_{\overline h}$ acts like $\sigma_{\delta +\overline g+\overline h},$ and so we compute
    \begin{align*}
        \tau_0\sigma_{\overline g}\tau_0\sigma_{\overline h}=\tau_0^2\sigma_{\overline g+\overline h}=\sigma_{\delta}\sigma_{\overline g+\overline h}=\sigma_{\delta+\overline g+\overline h}.
    \end{align*}  
\end{proof}

\begin{corollary}\label{cor: order of actions}Let $M$ be a NIM-rep over $\operatorname{GLM}(\Gamma,\delta)$.
     \begin{enumerate}[\upshape (a)]
        \item\label{item:i}  If $\delta=\overline{0}$, then $X_{\overline{h}}$ acts on the $2\Gamma$-orbits as a permutation of order 1 or 2.
        \item\label{item:ii}  If $\delta\neq\overline{0}$, then $X_{\overline{h}}$ acts on the $2\Gamma$-orbits  as a permutation of order 1, 2 or 4.
    \end{enumerate} 
\end{corollary}
\begin{proof}
   Since $\sigma_{g}^2=\text{id}$ for all $g\in \Gamma$, this follows directly from Proposition \ref{prop:tau0} and Lemma \ref{lemma:X-g in terms of sigmas}.
\end{proof}

\subsection{NIM-reps over \texorpdfstring{$\GLM$}{GLM} with a unique \texorpdfstring{$\Gamma$}{T}-orbit}

In this section, we provide a parametrization of irreducible NIM-reps over $\operatorname{GLM}(\Gamma, $ $\delta)$ with a unique $\Gamma$-orbit. We also classify these irreducible NIM-reps by showing when they are isomorphic.

Let $\mathcal M$ be the unique $\Gamma$-orbit, defined by a stabilizer subgroup $H<\Gamma$. Recall that, restricting to the $2\Gamma$-action, $\mathcal M$ partitions into $N:=2^m$ 2$\Gamma$-orbits $\mathcal O_1, \dots, \mathcal O_N$, for some $m\geq 0$. We follow the same notation as in Section \ref{sec:2Gamma action}.

\begin{theorem}\label{thm: 1 Gamma orbit}
   Irreducible NIM-reps over GLM$(\Gamma,\delta)$ with a unique $\Gamma$-orbit are parametrized by
   \begin{enumerate}[leftmargin=*,label=\rm{(\roman*)}] 
       \item a choice of subgroup $H<\Gamma$ such that $\sqrt{|2\Gamma|}$ divides $|H\cap 2\Gamma|$, and
       \item a choice of $\tau_0$ satisfying Proposition \ref{prop:tau0}.
   \end{enumerate}
\end{theorem}

\begin{proof}
Let $M$ be an irreducible NIM-rep with a unique $\Gamma$-orbit. Then we have a $\Gamma$-action on $M$. 
Take 
$H<\Gamma$ to be the subgroup that describes the (unique) orbit of this action, and let $\tau_0$ be induced from the corresponding action of $X_{\overline 0}$ on the set of $2\Gamma$-orbits. Then $\tau_0$ must satisfy Proposition \ref{prop:tau0}.  Let $m^{i,1}_r\in \mathcal O_i^1\subseteq \mathcal M^1.$ Then the summands of $X_{\overline g}\rhd m^{i,1}_r$ are contained in $\mathcal O_{\tau_0\sigma_g(i)}^1$. By Lemma \ref{lemma:unique-orbit-non-invertible-GLM}, we can rewrite Equation \eqref{eq:(1)} as 
    \begin{align*}
        |H\cap 2\Gamma|=(c_{(i,1),({\tau_0\sigma_g(i),1)}}^{\overline g})^2|2\Gamma:H \cap 2\Gamma|.
    \end{align*}
    Therefore, $c_{(i,1),({\tau_0\sigma_g(i)},1)}^{\overline g}=|H\cap 2\Gamma|/\sqrt{{|2\Gamma|}}$, which gives that $|H\cap 2\Gamma|/\sqrt{|2\Gamma|}$ is an integer.

Conversely, we need to check that any choice of $H<\Gamma$ such that $\sqrt{|2\Gamma|}$ divides $|H\cap 2\Gamma|$ and $\tau_0$ satisfying \ref{item:prop-i} and \ref{item:prop-ii} gives a well-defined irreducible NIM-rep.
A choice of $H<\Gamma$ defines the action of the invertibles in $\Gamma \subseteq \operatorname{GLM}(\Gamma, \delta).$ It remains to define the action on the non-invertible elements. 
By Proposition \ref{prop:tau0}, a choice of permutation $\tau_0$  satisfying \ref{item:prop-i} and \ref{item:prop-ii} gives a well defined action on the set of $2\Gamma$-orbits. Then, by Lemmas \ref{lemma: action on orbits-GLM}  and \ref{lemma:unique-orbit-non-invertible-GLM}, it remains to define the numbers $c_{(i,1),(k,1)}^{\overline g}$ for all $g\in \Gamma.$ Same as before, we must set $c_{i,{\tau_0\sigma_g(i)}}^{\overline g}:=|H\cap 2\Gamma|/\sqrt{{|2\Gamma|}}$, and the result follows. 
\end{proof}

 \begin{remark}
     For a NIM-rep with a unique $\Gamma$-orbit to exist, it is a necessary condition that $|2\Gamma|$ is a square.
 \end{remark}
 \begin{remark}
 It follows from the proof above that a NIM-rep over $\GLM$ with a unique $\Gamma$-orbit satisfies
 \begin{align*}
 &\textstyle	X_{\overline g}\rhd m_{r}^{i,1}=|H\cap 2\Gamma|/\sqrt{|2\Gamma|} \sum_{j=1}^N m_j^{\sigma_{g}\tau_0(i), 1}
 \end{align*}
    for all $\overline g\in \Gamma/2\Gamma$ and all  $i.$
 \end{remark}

					\begin{theorem}\label{thm: class unique Gamma orbit}
						Let $M=M( H, \tau_0)$ and $M'=M(H', \tau_0')$ be two irreducible NIM-reps  over GLM$(\Gamma,\delta)$ with a unique $\Gamma$-orbit. There exists a NIM-rep isomorphism $M\xrightarrow{\sim}M'$ if and only if $H=H'$ and $\tau_0=\tau_0'$.
					\end{theorem}

\begin{remark}\label{rk:reordering}
The definition of $\tau_0$ depends on a choice of ordering of the
$2\Gamma$-orbits of $M$, so the statement above should be interpreted
up to relabelling of these orbits.
More precisely, suppose $M(H,\tau_0)$ and $M(H',\tau_0')$ are two
NIM-reps which are isomorphic as $\Gamma$-sets. Then there exists a
$\Gamma$-equivariant bijection $\phi : M \to M'$. The map $\phi$
induces a bijection between the sets of $2\Gamma$-orbits of $M$ and
$M'$. After reordering the orbits of $M'$ if necessary, we may assume
that $\phi(\mathcal O_i)=\mathcal O'_i$ for all $i$, where
$\{\mathcal O_1,\dots,\mathcal O_n\}$ and
$\{\mathcal O'_1,\dots,\mathcal O'_n\}$ denote the respective sets of
$2\Gamma$-orbits.
With this identification of orbits, the corresponding parameters
satisfy $\tau_0=\tau_0'$.
\end{remark}
                    
					\begin{proof}
						Suppose there is an isomorphism $\psi: M\xrightarrow{\sim} M'$ as NIM reps. By restricting to the action of $\Gamma$, we get an induced isomorphism $M\simeq M'$ as $\Gamma$-sets, and so $H=H'.$ Now, since $\psi$ is an isomorphism as NIM-reps, it must respect the action of $X_{\overline{0}}$ in the set of $2\Gamma$-orbits, and so $\tau_0=\tau_0'$ follows. 
						
						Reciprocally, suppose that $H=H'$ and $\tau_0=\tau_0'$. Then we have  an isomorphism  $\psi: M\to M'$ of $\Gamma$-sets. To show that $\psi$ is an isomorphism of NIM-reps, it is enough to show that $X_{\overline{0}}\rhd \psi(m)=\psi(X_{\overline{0}}\rhd m)$, for all $m\in M$. Now, for $m\in \mathcal O_i,$ $X_{\overline{0}}\rhd \psi(m)$ is determined by $\tau_0'(i)$ and $\psi(X_{\overline{0}}\rhd m)$ is determined by $\tau_0(i)$. Since $\tau_0'=\tau_0$, then   $X_{\overline{0}}\rhd \psi(m)=\psi(X_{\overline{0}}\rhd m)$.
					\end{proof}

 \begin{corollary}\label{prop:1-orbit-GLM} Irreducible NIM-reps over $\operatorname{GLM}(\Gamma, \delta)$ with a unique $2\Gamma$-orbit are classified by subgroups $H<\Gamma$ such that
 $|H\cap 2\Gamma|/\sqrt{|2\Gamma|}$ is an integer and $|2\Gamma| /|H\cap 2\Gamma|= |\Gamma|/|H|$. Moreover, any two such NIM-reps obtained from different subgroups are non-isomorphic.
\end{corollary}
\begin{proof}  
At the level of group actions, having a unique $2\Gamma$-orbit means that the action of $2\Gamma$ is transitive on the orbit defined by $H$. This is the case if and only if $|2\Gamma| /|H\cap 2\Gamma|= |\Gamma|/|H|$. The parametrization then follows from Theorem \ref{thm: 1 Gamma orbit}, since we must take $\tau_0=\text{id}.$

Lastly, that any two such NIM-reps coming from different subgroups are non-isomorphic follows since $\Gamma$ is abelian, and so the group actions are non-isomorphic. 
\end{proof}

	\begin{example}\label{ex:Z2xZ2 1-orbit}
						Consider the case where $\Gamma=\mathbb{Z}_2\times \mathbb{Z}_2=\{(a,b):a,b\in \mathbb{Z}_2\}$. 
						Following Theorems \ref{thm: 1 Gamma orbit} and \ref{thm: class unique Gamma orbit}, we study the possibilities for NIM-reps over $\GLM$ with a unique $\Gamma$-orbit. When choosing $H=\Gamma$, we get a unique one-dimensional NIM-rep over $\GLM$. Proper subgroups of $\Gamma$ are  $F_1=\langle(1,0)\rangle$, $F_2=\langle (0,1)\rangle$, and $F_3=\langle (1,1)\rangle$; these could yield 2-dimensional NIM-reps with two $2\Gamma$-orbits, such that $X_{\overline{0}}$ acts on the $2\Gamma$-orbits as a permutation of order 1 or 2. To satisfy Property \ref{item:prop-i} of Proposition \ref{prop:tau0}, we can only choose $F_i$ such that $\delta=\overline{g}$ for $g\in F_i$. Thus when $\delta=\overline{(0,0)}$, any choice of $F_i$ works, and there are two possible choices of $\tau_0$ satisfying Proposition \ref{thm: 1 Gamma orbit} per order 2 subgroup: $(1 2)$ and $e$. 
                         Thus, when $\delta=\overline{(0,0)}$, there are six non-isomorphic NIM-reps coming from subgroups of order 2. 
						 For $\delta=\overline{g}$, with $g\ne \overline{(0,0)}$, we can only choose $F_i=\langle g\rangle $, and as before we get two non-isomorphic NIM-reps of dimension 2, coming from the same choices of $\tau_0$.
						
						Choosing $H=\{e\}$ could yield 4-dimensional NIM-reps with four $2\Gamma$-orbits. We will explore only the case when $\delta=\overline{(0,0)}$. Since the stabilizer subgroup is trivial, for all $g\ne (0,0)$, $\sigma_g$ is a permutation of order 2 that permutes all the elements of $M$. Without loss of generality, take $\sigma_{(1,0)}=(12)(34)$ 
                        and $\sigma_{(0,1)}=(13)(24)$. 
                        Possible choices of $\tau_0$  that satisfy Proposition \ref{prop:tau0} are $e, (1 2)(34), (13)(24),$ or $(14)(23)$; so by Theorem \ref{thm: class unique Gamma orbit} we have in total four non-isomorphic  irreducible NIM-reps over $\operatorname{GLM}(\mathbb Z_2\times \mathbb Z_2, \overline{(0,0)})$ of dimension four. A similar computation can be done for other choices of $\delta$.
			
					\end{example}

\subsection{NIM-reps over \texorpdfstring{$\GLM$}{GLM} with two \texorpdfstring{$\Gamma$}{T}-orbits} In this section, we provide a parametrization of irreducible NIM-reps over $\operatorname{GLM}(\Gamma, \delta)$ with two $\Gamma$-orbits, and classify them by showing when they are isomorphic. 

Let $\mathcal M^1$, $\mathcal M^2$ be the two $\Gamma$-orbits, defined by the stabilizer subgroups $H_1, H_2<\Gamma$, respectively. Restricting to the $2\Gamma$-action, $\mathcal M^i$ partitions into $N_i:=2^{m_i}$ $2\Gamma$-orbits, for some $m_i\geq 0$, $i\in\{1,2\}$.
Let $I=\{1, \dots, N_1\}$ and $J=\{1, \dots, N_2\}$ be the indices of the $2\Gamma$-orbits contained in $\mathcal M^1$ and $\mathcal M^2$, respectively.

\begin{lemma}\label{rk: switch of gamma orbits}
   For $M$ to be irreducible, $\tau_0\in S_{I\sqcup J}$ must satisfy for all $i\in I,$ $\tau_0(i)=j$ for some $j\in J$. In particular, this implies $N_1=N_2.$
\end{lemma}

\begin{proof}
    Suppose $\tau_0(i)=j$ for some $i,j\in I,$ and let $k\in I$. Since the action of $\Gamma$ in $\mathcal M^1$ is transitive, there exists $g\in \Gamma$  such that $\sigma_g(i)=k$. Then $\tau_0(k)=\tau_0(\sigma_g(i))=\sigma_g(\tau_0(i))=\sigma_g(j)$ which is again in $I$. This contradicts irreducibility, as the action of $X_{\overline{0}}$ must connect the $\Gamma$-orbits $\mathcal M^1$ and $\mathcal M^2$. 
\end{proof}

\begin{remark}\label{rk:no fixed points}
    It follows from the Lemma \ref{rk: switch of gamma orbits} that $X_{\overline 0}$ (and hence $X_{\overline g}$) cannot fix any $2\Gamma$- or $\Gamma$-orbit.
\end{remark}

Following the notation from Section \ref{sec:2Gamma action}, we have that $\sigma_{g}$ is a permutation in $S_{I}\times S_{J} < S_{I\sqcup J}$.

\begin{lemma}\label{prop:not-fix-orbits-GLM} 
    Suppose an irreducible NIM-rep $M$ over  $\GLM$ has two $\Gamma$-orbits, determined by the subgroups $H_1,H_2<\Gamma$, respectively. Then either
    \begin{enumerate}[leftmargin=*,label=\rm{(\roman*)}] 
        \item $\delta \in \overline{H_1}\cap \overline{H_2}$, or
        \item $\delta \not\in \overline{H_1}\cup \overline{H_2}$.
    \end{enumerate}
\end{lemma}
	\begin{proof}

					Assume for contradiction that $\delta\in \overline{H_1}$ and $ \delta\not\in \overline{H_2}$. 
				Then $\delta=\overline{g}$ for some $g\in H_1+2\Gamma$, and thus for any \(2\Gamma\)-orbit \(\mathcal O_i^1\) contained in the first \(\Gamma\)-orbit \(\mathcal M^1\) we have
					$\sigma_\delta(i)=i$. Now, by Lemma \ref{rk: switch of gamma orbits} for any $j\in J$ there exists $i_{j}$ in $I$ such that $\tau_0(i_{j})=j.$ Then 
					\begin{align*}
						j=   \tau_0\sigma_{\delta}(i_{j})=\sigma_{\delta}\tau_0(i_{j})=\sigma_{\delta}(j),
					\end{align*}
					and so $\sigma_{\delta}$ fixes all $j\in J$; in words, \(\sigma_\delta\) fixes every \(2\Gamma\)-orbit contained in \(\mathcal M^2\).
					
					Now fix any such \(2\Gamma\)-orbit \(\mathcal O_j^2\subset\mathcal M^2\) and pick a representative element \(m\in\mathcal O_j^2\). The fact that \(\sigma_\delta\) fixes the index \(j\) means \(g\cdot \mathcal O_j^2=\mathcal O_j^2\); equivalently
					\[g\rhd m \in \mathcal O_j^2 = 2\Gamma\cdot m.
					\]
					Hence there exists \(t\in 2\Gamma\) with \(g\rhd m = t\rhd m\). Rearranging,
					$(g-t)\rhd m = m,$
					so \(g-t\) lies in the stabilizer of \(m\). But the stabilizer of any element of \(\mathcal M^2\) is exactly \(H_2\). Therefore \(g-t\in H_2\), i.e.
				$g \in H_2 + 2\Gamma,$
					 contradicting the assumption \(\delta\notin\overline{H_2}\). Hence \(\delta\) either lies in both \(\overline{H_1}\) and \(\overline{H_2}\), or in neither, as claimed.
				\end{proof}

	In view of the previous lemma, either $\delta$ lies in the image of both stabilizers in $\Gamma/2\Gamma$, or in neither. Since $\delta\in\overline{H_i}$ is equivalent to $\sigma_\delta$ fixing all $2\Gamma$-orbits inside the $i$-th $\Gamma$-orbit, it follows that $\sigma_\delta$ either fixes all $2\Gamma$-orbits or fixes none of them.

\begin{corollary}\label{prop: action of tau0} Let $M$ be an irreducible NIM-rep over $\operatorname{GLM}(\Gamma, \delta)$ with two $\Gamma$-orbits. 
    \begin{enumerate}[\upshape (a)]
        \item\label{item:i-coro} If $\delta\in \overline{H_1}\cap \overline{H_2}$, then $\tau_{ 0}$ is a permutation of order 2 with no fixed points. 
        \item\label{item:ii-coro} If $\delta\not\in \overline{H_1}\cup \overline{H_2}$, then $\tau_{ 0}$ is a product of disjoint cycles of order 4 with no fixed points.
    \end{enumerate}
\end{corollary}
	\begin{proof}
			By Remark \ref{rk:no fixed points}, $\tau_0$ has no fixed points. By Proposition \ref{prop:tau0} we have $\tau_0^2=\sigma_{\delta}$.
			
			(a) If $\delta\in\overline{H_1}\cap\overline{H_2}$ then, as observed above, $\sigma_{\delta}$ fixes every $2\Gamma$-orbit; hence $\sigma_{\delta}=\mathrm{id}$. Thus $\tau_0^2=\mathrm{id}$, so every cycle of $\tau_0$ has length $1$ or $2$. Since $\tau_0$ has no fixed points, all cycles must have length $2$.
			
			(b) If $\delta\notin\overline{H_1}\cup\overline{H_2}$ then as observed above $\sigma_{\delta}$ fixes no $2\Gamma$-orbit. Recall also that each $\sigma_g$ satisfies $\sigma_g^2=\mathrm{id}$, so $\sigma_{\delta}$ has order \(1\) or \(2\). Combined with $\tau_0^2=\sigma_{\delta}$, this forces every cycle length \(k\) of \(\tau_0\) to satisfy \(k/\gcd(k,2)\in\{1,2\}\); hence \(k\in\{1,2,4\}\). The absence of fixed points rules out \(k=1\). If \(\tau_0\) had a 2-cycle then \(\tau_0^2\) would fix the two points in that cycle, contradicting that \(\sigma_{\delta}=\tau_0^2\) has no fixed points. Therefore \(k=4\) for every cycle, i.e. \(\tau_0\) is a product of disjoint 4-cycles.
		\end{proof}

To illustrate the result above, we include below a diagram of the action of $\GLM$ on the set of $2\Gamma$-orbits, showing how $X_{\overline{0}}$ arranges the said orbits in pairs when $\delta\in \overline {H_1} \cap \overline{H_2}$, and in sets of four when $\delta\not \in \overline {H_1} \cup \overline{H_2}$. 
\begin{table}[h]
          \centering
          \begin{tabular}{c|c}
  $\begin{tikzcd}
	&& {\mathcal O^2_{\sigma_g\tau_0(1)}} \\
	\\
	{\mathcal O_1^1 } && {\mathcal O_{\tau_0(1)}^2}
	\arrow[dashed, curve={height=-6pt}, from=1-3, to=3-1]
	\arrow[dashed, curve={height=-12pt}, from=3-1, to=1-3]
	\arrow[rightsquigarrow, curve={height=-6pt}, from=3-1, to=3-3]
	\arrow[color={rgb,255:red,39;green,190;blue,125}, from=3-3, to=1-3]
	\arrow[rightsquigarrow, curve={height=-6pt}, from=3-3, to=3-1]
\end{tikzcd}$     
& 
$\begin{tikzcd}
	&& {\mathcal O^2_{\sigma_g\tau_0(1)}} \\
	\\
	{\mathcal O_1^1 } && {\mathcal O_{\tau_0(1)}^2} \\
	{\mathcal O_{\tau_0^2(1)}^1} && {\mathcal O_{\tau_0^3(1)}^2} \\
	\\
	&& {\mathcal O_{\sigma_g\tau_0^3(1)}^2}
	\arrow[dashed, curve={height=-12pt}, from=1-3, to=4-1]
	\arrow[ dashed, curve={height=-12pt}, from=3-1, to=1-3]
	\arrow[rightsquigarrow, curve={height=-6pt}, from=3-1, to=3-3]
	\arrow[color={rgb,255:red,39;green,190;blue,125}, from=3-3, to=1-3]
	\arrow[rightsquigarrow, from=3-3, to=4-1]
	\arrow[rightsquigarrow, curve={height=6pt}, from=4-1, to=4-3]
	\arrow[dashed, curve={height=12pt}, from=4-1, to=6-3]
	\arrow[rightsquigarrow, from=4-3, to=3-1]
	\arrow[color={rgb,255:red,39;green,190;blue,125}, from=4-3, to=6-3]
	\arrow[dashed, curve={height=12pt}, from=6-3, to=3-1]
\end{tikzcd}$ 
\\ \hline 
               $\delta\in \overline {H_1} \cap \overline{H_2}$ & $\delta\not \in \overline {H_1} \cup \overline{H_2}$
          \end{tabular}
          \caption{Action on the set of $2\Gamma$-orbits, where dashed arrows denote action with $X_{\overline{g}}$, squiggle arrows action with $X_{\overline{0}}$ and green action with the group element $g$.}
          \label{tau_0}
      \end{table}

\begin{theorem}\label{thm:irrepNIMs-2Gammaorb}
  Irreducible NIM-reps over $\operatorname{GLM}(\Gamma,\delta)$ with two $\Gamma$-orbits are parametrized by 
   \begin{enumerate}[leftmargin=*,label=\rm{(\roman*)}] 
       \item\label{item-i-thm:irrepNIMs-2Gammaorb} A pair of subgroups $H_1, H_2<\Gamma$ satisfying:
       \begin{enumerate}[\upshape (a)]

           \item\label{item-i-a-thm:irrepNIMs-2Gammaorb} $\delta \in \overline{H_1}\cap \overline{H_2}$ or $\delta \notin \overline{H_1}\cup \overline{H_2}$,
           \item\label{item-i-b-thm:irrepNIMs-2Gammaorb} $|(\Gamma/H_1)[2]|=|(\Gamma/H_2)[2]|$,
           \item\label{item-iii-thm:irrepNIMs-2Gammaorb} $\sqrt{\frac{|H_1\cap 2\Gamma|\;|H_2\cap 2\Gamma|}{|2\Gamma|}} $ is an integer. 
       \end{enumerate}
       \item\label{item-ii-thm:irrepNIMs-2Gammaorb} A permutation $\tau_0$ of the $2\Gamma$-orbits satisfying
		Proposition~\ref{prop:tau0} and Lemma~\ref{rk: switch of gamma orbits}.
   \end{enumerate}
\end{theorem}

\begin{proof}
Let $M$ be an irreducible NIM-rep with two $\Gamma$-orbits determined by $H_1,H_2<\Gamma$.
		Condition (i)(a) follows from Lemma~\ref{prop:not-fix-orbits-GLM}.
		Condition (i)(b) follows from Lemma~\ref{rk: switch of gamma orbits}.
		The permutation $\tau_0$ must satisfy Proposition~\ref{prop:tau0} and Lemma~\ref{rk: switch of gamma orbits} by construction.
        
        It remains to prove (i)\ref{item-iii-thm:irrepNIMs-2Gammaorb}. 
Let $m^{i,1}_r\in \mathcal O_i^1\subseteq \mathcal M^1.$ Then the summands of $X_{\overline g}\rhd m^{i,1}_r$ are contained in $\mathcal O_{\tau_0\sigma_g(i)}^2$. Then we can rewrite Equation \eqref{eq:(1)} as 
\begin{align*}
	|H_1\cap 2\Gamma|=(c_{(i,1),({\tau_0\sigma_g(i),2)}}^{\overline g})^2|2\Gamma:H_2 \cap 2\Gamma|.
\end{align*}
Solving gives $$c_{(i,1),({\tau_0\sigma_g(i),2)}}^{\overline g}=\sqrt{|H_1\cap 2\Gamma||H_2\cap 2\Gamma|/{|2\Gamma|}},$$ which must be an integer.

Conversely, we need to check that any choice of $H_1, H_2<\Gamma$ 
and $\tau_0$ satisfying \ref{item-i-thm:irrepNIMs-2Gammaorb}--\ref{item-ii-thm:irrepNIMs-2Gammaorb}
give a well-defined NIM-rep with two $\Gamma$-orbits.
The subgroups $H_1, H_2<\Gamma$ define the $\Gamma$-action. It remains to extend this to an action of the whole fusion ring $\GLM$. 
First, since $\tau_0$ satisfies  Proposition \ref{prop:tau0}, we have a well-defined action of $\GLM$ on the induced set of $2\Gamma$-orbits. 
Then by Lemmas \ref{lemma: action on orbits-GLM} and \ref{lemma:X-g in terms of sigmas} it remains to define the numbers  $c_{(i,1),({\tau_0\sigma_g(i),2)}}^{\overline g}$, for all $g\in \Gamma.$ Again, by Equation \eqref{eq:(1)}, we must set  $c_{(i,1),({\tau_0\sigma_g(i),2)}}^{\overline g}=\sqrt{|H_1\cap 2\Gamma||H_2\cap 2\Gamma|/{|2\Gamma|}}$, and this gives a well-defined NIM-rep. 
Lastly, since $\tau_0$ satisfies Lemma \ref{rk: switch of gamma orbits}, the NIM-rep is irreducible, as desired.
\end{proof}

\begin{remark}
		It follows from the proof above that an irreducible NIM-rep over $\GLM$ with two $\Gamma$-orbits satisfies
		\begin{align*}
			&\textstyle	X_{\overline g}\rhd m_{r}^{i,\ell_1}=\sqrt{|H_1\cap 2\Gamma||H_2\cap 2\Gamma|/{|2\Gamma|}}\sum_{j=1}^N m_j^{\sigma_{g}\tau_0(i), \ell_2}
		\end{align*}
		for all $\overline g\in \Gamma/2\Gamma$ and $\ell_1\ne \ell_2\in \{1,2\}$.
\end{remark}

Given a choice of $H_1,H_2<\Gamma$ and permutation $\tau_0$ as in Theorem \ref{thm:irrepNIMs-2Gammaorb}, we denote the corresponding irreducible NIM-rep by $M(H_1, H_2, \tau_0).$

	\begin{theorem}\label{thm: class two gamma orbits}
					The NIM-reps $M(H_1, H_2, \tau_0)$ and $M(H'_1, H'_2, \tau_0')$ 
                    are isomorphic if and only if (up to reordering) $H_1=H_1'$, $H_2=H_2'$  and $\tau_0=\tau_0'.$
				\end{theorem}

\begin{remark}Same as in Remark \ref{rk:reordering},
the definition of $\tau_0$ depends on a choice of ordering of the
$2\Gamma$-orbits of $M$. Thus equality $\tau_0=\tau_0'$ is understood
after relabeling the corresponding $2\Gamma$-orbits.
\end{remark}

              	\begin{proof}
			Write $M=M(H_1,H_2,\tau_0)$ and $M'=M(H'_1,H'_2,\tau_0')$ and denote by
			$\mathcal M^1,\mathcal M^2$ (resp.\ $\mathcal M'^1,\mathcal M'^2$)
			the two $\Gamma$-orbits of $M$ (resp.\ $M'$).  
			Assume there is an isomorphism of NIM-reps
			$\Psi\colon M\to M'$. Forgetting the action of noninvertible basis
			elements gives a $\Gamma$-set isomorphism $\psi$ between the underlying
			$\Gamma$-sets. Hence $\psi$ either preserves the two $\Gamma$-orbits or
			exchanges them; in particular the stabilizers agree up to reordering; therefore, we may assume
			$H_i=H'_i$ for $i=1,2$.
			Next, observe that $\Psi$ must carry each $2\Gamma$-orbit of $M$ to a
			$2\Gamma$-orbit of $M'$; after relabeling the $2\Gamma$-orbits of $M'$ assume
			$\psi(\mathcal O_i^j)=\mathcal O'{}^j_i$ for every $i,j$.
				Finally, because $\Psi$ must respect the action of all elements in the fusion ring,
			it must, in particular, respect the action of $X_{\overline 0}$ on the
			set of $2\Gamma$-orbits. Under the chosen identification of
			$2\Gamma$-orbits this exactly says that $\tau_0=\tau_0'$.

			Conversely, suppose $H_i=H'_i$, for $i=1,2$, and  $\tau_0=\tau_0'$. We will construct an
			isomorphism of NIM-reps $\Psi\colon M\to M'$. Because the stabilizers coincide, there is a canonical isomorphism of
			$\Gamma$-sets $\psi\colon M\to M'$ which identifies each
			$\Gamma$-orbit $\mathcal M^j$ with $\mathcal M'^j$ and each
			$2\Gamma$-orbit $\mathcal O_i^j$ with $\mathcal O'{}^j_i$. By
			construction $\psi$ commutes with the $\Gamma$-action.
				To upgrade $\psi$ to an isomorphism of NIM-reps, we must check that it also
		respects the action of every non-invertible basis element
			$X_{\overline g}$. By the fusion rules, we have
			$X_{\overline g}=g\cdot X_{\overline 0}$, so it suffices to check
			for $X_{\overline 0}$.
			Fix a $2\Gamma$-orbit $\mathcal O^1_i\subset\mathcal M^1$ and let
			$m\in\mathcal O^1_i$. By parameterization of the NIM-rep,
			$X_{\overline 0}\rhd m$ is supported exactly on the single
			$2\Gamma$-orbit $\mathcal O^2_{\tau_0(i)}$, with each basis element in
			that orbit appearing with multiplicity
			\[
			\kappa:=\sqrt{\frac{|H_1\cap 2\Gamma|\;|H_2\cap 2\Gamma|}{|2\Gamma|}}.
			\]
			Because $H_i=H'_i$, the same multiplicity $\kappa$ is used in the
			definition of $M'$. Since $\psi(\mathcal O^2_{\tau_0(i)})=
			\mathcal O'{}^2_{\tau_0(i)}$ and $\tau_0=\tau_0'$, it follows that
			\[
			\psi\bigl(X_{\overline 0}\rhd m\bigr)
			= X_{\overline 0}\rhd \psi(m),
			\] as desired.
			Therefore, $\psi$ is an isomorphism of NIM-reps $M\cong M'$, as
			required.
		\end{proof}

\begin{example}
    We follow up on the case $\Gamma=\mathbb{Z}_2 \times \mathbb{Z}_2$, described for the 1-orbit case in Example \ref{ex:Z2xZ2 1-orbit}, and use the same notation for the subgroups of $\Gamma$ there. We fix $\delta=\overline{(0,0)}$. In order to compute the irreducible NIM-reps for two $\Gamma$-orbits, by Theorem \ref{thm:irrepNIMs-2Gammaorb}~\ref{item-i-thm:irrepNIMs-2Gammaorb}, we need to first choose two subgroups from the list $\lbrace \Gamma, F_1, F_2, F_3,\lbrace \overline{(0,0)}\rbrace \rbrace$. The choice of subgroups is subject to Theorem \ref{thm:irrepNIMs-2Gammaorb}~\ref{item-i-thm:irrepNIMs-2Gammaorb}\ref{item-i-b-thm:irrepNIMs-2Gammaorb}, this numerical condition tells us that there are only three possible cases: when both subgroups are $\Gamma$ (Case 1, where the NIM-rep is 2-dimensional), when both subgroups are either $F_1$, $F_2$, or $F_3$ (Case 2, the NIM-rep is 4-dimensional) or when both subgroups are $\lbrace \overline{(0,0)}\rbrace$ (Case 3, the NIM-rep is 8-dimensional). Note that for these three cases, Condition \ref{item-i-thm:irrepNIMs-2Gammaorb}\ref{item-iii-thm:irrepNIMs-2Gammaorb}
of the Theorem \ref{thm:irrepNIMs-2Gammaorb} follows immediately. 
    \begin{itemize}
         \item[-] Case 1: here, by Corollary \ref{prop: action of tau0}, we know that $\tau_0$ is $\left( 12  \right)$. Regarding Proposition \ref{prop:tau0}, both conditions will be satisfied immediately, as $\sigma_g=e$ for all $g$. Then we obtain  a unique two-dimensional NIM-rep with two $\Gamma$-orbits.
        \item[-] Case 2: in this case, we have two $\Gamma$-orbits, and two $2 \Gamma$-orbits in each $\Gamma$-orbit. 
       It is a quick check that choosing different subgroups is not possible, as there is no suitable $\tau_0$ that satisfies Proposition \ref{prop:tau0}. Then we assume that both of our subgroups are $F_i$ for $i=1,2$ or 3. Without loss of generality, we label by $\{\mathcal O_1, \mathcal O_2\}$ and $\{\mathcal O_3, \mathcal O_4\}$ the 2$\Gamma$-orbits in each $\Gamma$-orbit. Then $\sigma_g=e$ if $g\in H_i$ and $\sigma_g=(1 2)(3 4)$ if $g\notin H_i$. It follows that there are two possible choices for $\tau_0$, given by $( 1 3)(2 4)$ and $(1 4)(2 3)$, respectively. Hence, by Theorems \ref{thm:irrepNIMs-2Gammaorb} and \ref{thm: class two gamma orbits}, we get in total 6 non-isomorphic NIM-reps, two for each choice of $F_i$.
        \item[-] Case 3: in this case, we have two $\Gamma$-orbits and four one-dimensional $2 \Gamma$-orbit per $\Gamma$-orbit. Without loss of generality, up to relabeling we choose $\sigma_{(1,0)}=\left( 12 \right) \left( 34 \right) \left( 56 \right) \left( 78 \right)$ (where $\lbrace 1,2,3,4 \rbrace$ label the $2 \Gamma$-orbits in one $\Gamma$-orbit and $\lbrace 5,6,7,8 \rbrace$ the other four $2 \Gamma$-orbits), and $\sigma_{(0,1)}=\left( 13 \right) \left( 24 \right) \left( 57 \right) \left( 68 \right)$. In this case, $\tau_0$ can only have the shape $\left( 1 i_1 \right) \left( 2 i_2 \right) \left( 3 i_3 \right) \left( 4 i_4 \right)$ (where $i_1,i_2,i_3,i_4 \in \lbrace 5,6,7,8 \rbrace$). Checking Proposition \ref{prop:tau0} \ref{item:prop-ii}, we see that the only options for $\tau_0$ are: $\left( 1 5 \right) \left( 2 6 \right) \left( 3 7 \right) \left( 4 8 \right)$, $\left( 16 \right) \left( 2 5 \right) \left( 3 8 \right) \left( 4 7 \right)$, $\left( 1 7 \right) \left( 2 8 \right) \left( 35 \right) \left( 4 6\right)$, and $\left( 18 \right) \left( 2 7 \right)\left( 3 6 \right) \left( 4 5 \right)$, from which we obtain four non-isomorphic NIM-reps in this case. 
    \end{itemize}

\end{example}

\subsection{Algebra objects from \texorpdfstring{$\operatorname{GLM}(\Gamma,\delta)$}{GLM}} 
In this section, we use our previous classification results for irreducible NIM-reps over $\GLM$ to detect potential algebra objects.

\begin{proposition}
    All irreducible NIM-reps over $\GLM$ are admissible.
\end{proposition}
\begin{proof}
     If we have a unique $\Gamma$-orbit, then all objects in $M$ are connected by the action of $\Gamma$, and so the NIM-rep is admissible.  Because the orbits are connected by irreducibility of the NIM-reps, we can repeat a similar argument for the case of two $\Gamma$-orbits, and so every irreducible NIM-rep classified at Theorem \ref{thm:irrepNIMs-2Gammaorb} is admissible.
\end{proof}

 \begin{proposition}
Algebra objects associated to an irreducible NIM-rep over $\GLM$ with a unique $\Gamma$-orbit are of the form
     \[A=\oplus_{h \in H} h \oplus_{i=1}^n\left(\oplus_{\substack{g \in \Gamma:\\\sigma_g(i)=\tau_0^{-1}(i)} } c_i^g X_{\overline{g}}\right).\]
 \end{proposition}
\begin{proof}
If the irreducible NIM-rep $M$ over GLM$(\Gamma,\delta)$ has a unique $\Gamma$-orbit, then $X_{\overline{g}}$ has a loop at $m\in M$ if and only if $X_{\overline{g}}$ fixes the $2\Gamma$-orbit $\mathcal{O}_i$, where $m\in \mathcal{O}_i$. Moreover, the number of self loops is $c^g_i$. So, $\sigma_g\tau_0(i)=i$ if and only if $\sigma_g(i)=\tau_0^{-1}(i)$. Thus we obtain the  algebra objects as in the statement.
\end{proof}

\begin{proposition} Algebra objects associated to an irreducible NIM-rep over $\GLM$ with two $\Gamma$-orbits are of the form 
    \[A_i=\oplus_{h \in H_i} {h}  \]
     for $i=1,2.$
 \end{proposition}
\begin{proof}
If the irreducible NIM-rep $M$ over GLM$(\Gamma,\delta)$ has two $\Gamma$-orbits, then, by Remark \ref{rk:no fixed points}, $X_{\overline{0}}$ does not fix any $2\Gamma$-orbit, and  neither does $X_{\overline{g}}$. Therefore, we obtain the algebra objects as in the statement.
\end{proof}

\bibliography{womap3-bib}
\bibliographystyle{alpha}

\end{document}